\documentclass[10pt,reqno]{amsart}

\usepackage[dvipsnames]{xcolor}
\usepackage{amsmath,amssymb,amsthm,amsfonts,enumerate,tikz,bm}
\usepackage{mathrsfs}
\usetikzlibrary{matrix,arrows,positioning,calc}
\usepackage[all]{xy}
\usepackage[bookmarksnumbered,colorlinks]{hyperref}
\usepackage{url}
\numberwithin{equation}{section}
\usepackage{graphicx}
\usepackage[T1]{fontenc}
\usepackage{textcomp}
\usepackage{palatino,helvet}

\usepackage{footmisc}

\newtheorem{definition}{Definition}[section]
\newtheorem{remark}[definition]{Remark}
\newtheorem{example}[definition]{Example}

\newtheorem{theorem}[definition]{Theorem}
\newtheorem{proposition}[definition]{Proposition}
\newtheorem{lemma}[definition]{Lemma}
\newtheorem{corollary}[definition]{Corollary}
\theoremstyle{remark}

\usepackage{paralist}

\usepackage{scrextend}
\deffootnote{2em}{0em}{\thefootnotemark\quad}

\newcommand{\mbE}{\mathbb{E}}
\newcommand{\mcA}{\mathcal{A}}
\newcommand{\mcB}{\mathcal{B}}
\newcommand{\mcC}{\mathcal{C}}
\newcommand{\mcD}{\mathcal{D}}

\newcommand{\mcF}{\mathcal{F}}
\newcommand{\mcG}{\mathcal{G}}
\newcommand{\mcI}{\mathcal{I}}

\newcommand{\mcP}{\mathcal{P}}
\newcommand{\mcS}{\mathcal{S}}
\newcommand{\mcX}{\mathcal{X}}
\newcommand{\mcY}{\mathcal{Y}}
\newcommand{\mcT}{\mathcal{T}}

\newcommand{\mfC}{\mathfrak{C}}
\newcommand{\mfF}{\mathfrak{F}}
\newcommand{\mfP}{\mathfrak{P}}

\newcommand{\modu}{\mathrm{mod}}
\newcommand{\add}{\mathrm{add}}

\newcommand{\Mod}{\mathsf{Mod}}
\newcommand{\fgMod}{\mathsf{mod}}

\newcommand{\Thick}{\mathsf{Thick}}

\newcommand{\resdim}{\mathrm{resdim}}
\newcommand{\coresdim}{\mathrm{coresdim}}
\newcommand{\pd}{\mathrm{pd}}
\newcommand{\id}{\mathrm{id}}

\newcommand{\Hom}{\mathrm{Hom}}
\newcommand{\Ext}{\mathrm{Ext}}

\makeatletter
\def\@seccntformat#1{%
  \protect\textup{\protect\@secnumfont
    \ifnum\pdfstrcmp{section}{#1}=0 \scshape\bfseries\fi
    \ifnum\pdfstrcmp{subsection}{#1}=0 \bfseries\fi
    \csname the#1\endcsname
    \protect\@secnumpunct
  }%
}
\makeatother

\begin{document}

\title{Cut notions in extriangulated categories related to Auslander-Buchweitz theory and cotorsion theory}

\author{Mindy Huerta}
\address[M. Huerta]{Facultad de Ciencias. Universidad Nacional Aut\'onoma de M\'exico. Circuito Exterior, Ciudad Universitaria. CP04510. Mexico City, MEXICO}
\email{mindyhp90@ciencias.unam.mx}

\author{Octavio Mendoza}
\address[O. Mendoza]{Instituto de Matem\'aticas. Universidad Nacional Aut\'onoma de M\'exico. Circuito Exterior, Ciudad Universitaria. CP04510. Mexico City, MEXICO}
\email{omendoza@matem.unam.mx}

\author{Corina S\'aenz}
\address[C. S\'aenz]{Facultad de Ciencias. Universidad Nacional Aut\'onoma de M\'exico. Circuito Exterior, Ciudad Universitaria. CP04510. Mexico City, MEXICO}
\email{ecsv@ciencias.unam.mx}

\author{Valente Santiago}
\address[V. Santiago]{Facultad de Ciencias. Universidad Nacional Aut\'onoma de M\'exico. Circuito Exterior, Ciudad Universitaria. CP04510. Mexico City, MEXICO}
\email{valente.santiago.v@gmail.com}

\maketitle

\begin{abstract}
In this work we introduce notions in Auslander-Buchweitz theory and cotorsion theory in extriangulated
categories which extend the given ones for abelian categories. Although these notions have been already developed for extriangulated categories in remarkable works, in this paper, we tackle them in a relative sense by considering subcategories of objects. 
This approach not only covers the existing theory given on abelian, exact and triangulated categories, but it also shows how to get similar results with an appropriate treatment of
local properties. 
\end{abstract}


\pagestyle{myheadings}
\markboth{\rightline {\scriptsize M. Huerta, O. Mendoza, C. S\'{a}enz and V. Santiago}}
         {\leftline{\scriptsize Cut notions in extriangulated categories}}

\section*{\textbf{Introduction}}

An important branch of relative homological algebra was developed
by M. Auslander and R. O. Buchweitz in \cite{ABtheory}.
This theory, called
\emph{Auslander-Buchweitz approximation theory (AB-theory, for short)}, consists of methods for obtaining right and left approximations from (co)generators of a full subcategory of an
abelian category. Thus, computing left and right approximations is a topic of interest in representation theory and homological algebra, and hence the study of techniques on
how to get these approximations by using AB-theory became relevant at this point.
Later, a technique which 
has an appealing relation with
left and right approximations, emerged. 
In \cite{GT, EJ1, EJ2} it is shown that is possible to produce left and
right approximations via \emph{complete cotorsion pairs}. Since then several generalizations have 
been given of this notion, see for instance \cite{HMP, HMPcut}. 

Since the approximation of objects is a common matter between AB-theory and cotorsion pairs theory, it is quiet natural to ask for the possible relationship between them. This question
has generated, in recent years, works
describing such interplay. An example of this can be found in \cite{BMPS} 
where Frobenius pairs and weak Auslander-Buchweitz contexts are defined and a relativization of cotorsion pairs under the name of \emph{$\mcS$-cotorsion pairs} is proposed. Moreover, it is also shown in \cite{BMPS} that AB-contexts and $\mcS$-cotorsion pairs are in one-to-one correspondence. From this approach of 
relativizing cotorsion pairs, the 
notion of \emph{cotorsion pair cut along $\mcS$} is introduced in \cite{HMPcut}.
Moreover, it is also extended the correspondence theorem for this new concept \cite[Thms. 4.6 \& 4.12]{HMPcut}.

The scope of AB-theory and cotorsion pairs is not limited to the setting of abelian categories, it can be extended to more general settings such as triangulated categories, see 
\cite{MendozaSaenzVargasSouto2} and \cite{MendozaSaenzVargasSouto}. In 
\cite{Nakaoka1}, H. Nakaoka and Y. Palu introduce the notion of extriangulated categories, which
extends well-known results on abelian, exact and triangulated categories. It is also known that AB-theory and cotorsion theory
do not scape to generalizations on this context as we can see in the remarkable works done in 
\cite{MendozaSaenzVargasSouto2, MDZtheoryAB, Hegorensteinobjects, TGone} and \cite{LNheartsoftwin, HZontherelation} for such theories, respectively.

The first goal of this work is to give a \emph{cut version} of AB-theory in extriangulated categories. In this case, "cut version" means that we adapt some existing notions of this theory by considering a relativization which depends on a class of objects $\mcS$ in an
extriangulated category $\mcC$. The purpose is also to show that some outcomes already given in the 
absolute case can be covered by taking $\mcS=\mcC$. With this in mind, the second goal is to generalize the notion of cotorsion pair in extriangulated categories
through the concept of \emph{cut $n$-cotorsion pair}. This new concept has the advantage that 
generalizes $n$-cotorsion pairs \cite[Def. 2.1]{HMP} and cut cotorsion pairs \cite[Def. 2.1]{HMPcut} at the same time. Hence, following the ideas in \cite{HMPcut} for abelian categories, it will be useful to show that there is an interplay with
the cut notions developed in the first part.

\subsection*{Organization of the paper} 

In Section 1, we recall some well-known results in extriangulated categories appearing mainly in \cite{Nakaoka1} and establish the notation that will be used for the rest of this work. In Section 2,
we carry outcomes (given originally in \cite{HMPcut} for abelian categories) to extriangulated categories. As a first contribution and for the purpose of this work, we develop in detail the proofs of some of them. Section 3 is devoted to present generalizations of left Frobenius pairs and Auslander-Buchweitz contexts called \emph{cut Frobenius 
pair} and \emph{cut Auslander-Buchweitz context} (Definitions~\ref{def: cut left Frobenius pair} 
and~\ref{def: cut left AB context}), we provide examples of such concepts and finish the section by 
proving that there exists a one-to-one correspondence between them (Theorem~\ref{theo:correspondence_1}). It is worth mentioning that, even when the proofs were done firstly for abelian categories and these
proofs can be easily adapted to extriangulated categories, in this section we rewrite and give 
alternative proofs for some parts of them. In Section 4, we introduce the notion of 
\emph{cut $n$-cotorsion pair} (Definition~\ref{def: CnCP ext}) in order to unify the concepts of cut cotorsion pairs and
$n$-cotorsion pairs for extriangulated categories. We also provide several examples in abelian, 
triangulated and extriangulated categories (Examples~\ref{ex:carcaj},~\ref{ex: sistest} and~\ref{ex: sistest ext}). In particular, we show in Example~\ref{ex: GP, P2} that cut cotorsion pairs are different to 
cotorsion pairs in extriangulated categories whose extriangulated structure is induced by subcategories closed under extensions. Furthermore, we give in Proposition~\ref{pro: 1cot<->complete cot} sufficent and neccesary conditions so that the equality holds. In the last part, we give an example
of an $n$-cotorsion pair in an extriangulated category which is neither exact nor triangulated 
(Example~\ref{ex: ni exact ni triang}). In Section 5 we give results that involved cut $1$-cotorsion 
pairs and Auslander-Buchweitz contexts in order to prove the most important result of this section, Theorem~\ref{theo:correspondence_2}. As a consequence of this, by considering a restriction on 
representatives of such correspondence, we also get Corollary~\ref{cor: leftFrob<->cotor pairs in Thick}. Finally, Section 6 is devoted to show how the notions introduced in the previous sections can
be applied in the case of triangulated categories and the so-called  
co-$t$-structures (Theorems~\ref{thm: t-str y cut cot} and~\ref{thm: co-t-st y cut cot}).


\subsection*{Conventions}

Throughout the paper, $\mcC$ denotes an additive category. Among the main examples considered in this work are the following ones:
\begin{itemize}
\item $\Mod(R)$ which is the category of left $R$-modules over an associative ring $R$ with identity. 

\item $\fgMod(\Lambda)$ which is the category of finitely generated left $\Lambda$-modules over an Artin algebra $\Lambda$.  
\end{itemize}

We write $\mcS \subseteq \mcC$ to denote that $\mcS$ is a full subcategory of $\mcC$ or a class of objects in $\mcC$. All the subcategories of $\mcC$ are assumed to be full. Given $X, Y \in \mcC$, we denote by $\mcC(X,Y)$ the group of morphisms from $X$ to  $Y$. Monomorphisms and epimorphisms in $\mcC$ are denoted by using the arrows $\rightarrowtail$ and $\twoheadrightarrow$, respectively. In case $X$ and $Y$ are isomorphic, we write $X \simeq Y$. The notation $F \cong G$, on the other hand, is reserved to denote the existence of a natural isomorphism between functors $F$ and $G$.


\section{\textbf{Preliminaries}}\label{sec:preliminaries}

\subsection*{Extriangulated categories and terminology}

We begin with some definitions and results related to extriangulated categories. 
For a detailed treatise on this matter we refer to \cite{Nakaoka1, LNheartsoftwin, MDZtheoryAB}. We recall that $\mcC$ denotes an additive category.

\begin{definition}\cite[Def. 2.1 \& Rmk. 2.2]{Nakaoka1}
Let $\mathbb{E}:\mcC^{op}\times \mcC\to \mathrm{Ab}$ be an additive bifunctor. An $\mathbb{E}$-extension is a triplet $(A, \delta, C),$ where $A, C\in \mcC$ and $\delta\in \mathbb{E}(C, A).$ 
For any $a\in \mcC(A, A')$
and $c\in \mcC(C', C)$, we have $\mathbb{E}$-extensions $a\delta:=\mathbb{E}(C, a)(\delta)\in \mathbb{E}(C, A')$ and
$\delta c:=\mathbb{E}(c, A)(\delta)\in \mathbb{E}(C', A)$. In this terminology, we have
$(a\delta)c=a(\delta c)$ in $\mathbb{E}(C', A')$.
\end{definition}

\begin{definition}\cite[Def. 2.3]{Nakaoka1}
Let $(A, \delta, C)$ and  $(A', \delta', C')$ be $\mathbb{E}$-extensions. A morphism $(a, c): 
(A, \delta, C)\to (A', \delta', C')$ of $\mathbb{E}$-extensions is a pair of morphisms $a\in \mcC(A, A')$ and
$c\in \mcC(C, C')$ in $\mcC$, satisfying the equality $a\delta=\delta' c.$
We simply denote it as $(a, c): \delta\to \delta'$.
We obtain the category $\mathbb{E}$-$\Ext(\mcC)$ of $\mathbb{E}$-extensions, with composition and identities
naturally induced by those in $\mcC$.
\end{definition}

\begin{definition}\cite[Def. 2.5]{Nakaoka1}
For any $A, C\in \mcC$, the zero element $0\in \mathbb{E}(C, A)$ is called the split $\mathbb{E}$-extension.
\end{definition}

\begin{definition}\cite[Def. 2.6]{Nakaoka1}
Let $\delta=(A, \delta, C)$ and $\delta'=(A', \delta', C')$ be $\mathbb{E}$-extensions. Let 
$C\mathop{\to}\limits^{\iota_{C}} C\oplus C\mathop{\leftarrow}\limits^{\iota_{C'}} C'$ and 
$A\mathop{\leftarrow}\limits^{p_{A}} A\oplus A'\mathop{\to}\limits^{p_{A'}}A'$ be the coproduct and product in $\mcC$,
respectively. We remark that, by the biadditivity of $\mathbb{E}$, we have a natural isomorphism
$$\mathbb{E}(C\oplus C', A\oplus A')\cong \mathbb{E}(C, A)\oplus \mathbb{E}(C, A')\oplus 
\mathbb{E}(C', A)\oplus \mathbb{E}(C', A').$$
Let $\delta\oplus \delta'\in \mathbb{E}(C\oplus C', A\oplus A')$ be the element corresponding to $(\delta, 0, 0, 
\delta')$ through this isomorphism.

If $A=A'$ and $C=C'$, then the sum $\delta+\delta'\in \mathbb{E}(C, A)$ of $\delta, \delta'\in 
\mathbb{E}(C, A)$ is obtained by 
$$\delta+\delta'=\nabla_{A}(\delta\oplus \delta')\Delta_{C}$$
where $\Delta_{C}=\left( 
\begin{array}{c}
1_C\\1_C
\end{array}
\right) : C\to C\oplus C$ and $\nabla_{A}=(1_A\, 1_A): A\oplus A\to A$.
\end{definition}

\begin{definition}\cite[Def. 2.7]{Nakaoka1}
Let $A, C\in \mcC$ be any pair of objects. Two sequences of morphisms $A\mathop{\to}\limits^{x} B
\mathop{\to}\limits^{y} C$ and $A\mathop{\to}\limits^{x'} B'\mathop{\to}\limits^{y'} C$ in $\mcC$ are said to be
equivalent if there exists an isomorphism $b\in \mcC(B, B')$ which makes the following diagram commutative
\[
\xymatrix@R=4mm{
& B\ar[dr]^{y}\ar[dd]^{b}_{\wr} &\\
A\ar[ur]^{x}\ar[dr]_{x'} & & C.\\
& B'\ar[ur]_{y'} & 
}
\]
We denote the equivalence class of $A\mathop{\to}\limits^{x} B\mathop{\to}\limits^{y} C$ by 
$[A\mathop{\to}\limits^{x} B\mathop{\to}\limits^{y} C]$.
\end{definition}

\begin{definition}\cite[Def. 2.8]{Nakaoka1}
For any two classes $[A\mathop{\to}\limits^{x} B\mathop{\to}\limits^{y} C]$ and
$[A'\mathop{\to}\limits^{x'} B'\mathop{\to}\limits^{y'} C']$, we set
$[A\mathop{\to}\limits^{x} B\mathop{\to}\limits^{y} C]\oplus [A'\mathop{\to}\limits^{x'} B'\mathop{\to}\limits^{y'} C']:=[A\oplus A'\mathop{\to}\limits^{x\oplus x'} B\oplus B'\mathop{\to}\limits^{y\oplus
y'} C\oplus C'].$
\end{definition}

\begin{definition}\cite[Def. 2.9]{Nakaoka1}\label{def 2.9}
Let $\mathfrak{s}$ be a correspondence which associates to each $\mathbb{E}$-extension $\delta\in \mathbb{E}(C, A)$ an equivalence class $\mathfrak{s}(\delta)=
[A\mathop{\to}\limits^{x} B\mathop{\to}\limits^{y} C]$. This $\mathfrak{s}$ is called a realization of $\mathbb{E}$ if it satisfies the following condition:

$(*)$ Let $\delta\in \mathbb{E}(C, A)$ and $\delta'\in \mathbb{E}(C', A')$ be any pair of $\mathbb{E}$-extensions, with $\mathfrak{s}(\delta)=[A\mathop{\to}\limits^{x} B\mathop{\to}\limits^{y} C]$ and
$\mathfrak{s}(\delta')=[A'\mathop{\to}\limits^{x'} B'\mathop{\to}\limits^{y'} C']$. Then, for any morphism 
$(a, c)\in \mathbb{E}\mbox{-}\Ext(\mcC)(\delta, \delta')$, there exists $b\in \mcC(B, B')$ which makes the 
following diagram commutative
\begin{equation}\label{eq: diag1}
\xymatrix{
A\ar[r]^{x}\ar[d]_{a} & B\ar[r]^{y}\ar[d]^{b} & C\ar[d]^{c}\\
A'\ar[r]_{x'} & B'\ar[r]_{y'} & C'.
}
\end{equation}
We say that the sequence $A\mathop{\to}\limits^{x} B\mathop{\to}\limits^{y} C$ realizes $\delta$,
whenever it satisfies $\mathfrak{s}(\delta)=[A\mathop{\to}\limits^{x} B\mathop{\to}\limits^{y} C]$. We remark that this condition does not depend on the choices of the representatives of the equivalence classes. 
In the above situation, we say that \eqref{eq: diag1} (or the triplet $(a, b, c)$) realizes $(a, c)$.
\end{definition}

\begin{definition}\cite[Def. 2.10]{Nakaoka1}
A realization $\mathfrak{s}$ of $\mathbb{E}$ is additive if it satisfies the following two conditions:
\begin{enumerate}
\item $\mathfrak{s}(0)=\left[A\mathop{\to}\limits^{\tiny{\left(
\begin{array}{c}
1\\ 0
\end{array}
\right)}} A\oplus C\mathop{\to}\limits^{(0\, 1)} C
\right]$ for any $A, C\in \mcC;$
\item $\mathfrak{s}(\delta\oplus \delta')=\mathfrak{s}(\delta)\oplus \mathfrak{s}(\delta')$, for
any $\mathbb{E}$-extensions $\delta$ and $\delta'$.
\end{enumerate}
\end{definition}

\begin{definition}\cite[Def. 2.12]{Nakaoka1}
The pair $(\mathbb{E}, \mathfrak{s})$ is an external triangulation of $\mcC$ if it satisfies the following
conditions.
\begin{enumerate}
\item[(ET1)] $\mathbb{E}: \mcC^{op}\times \mcC\to \mathrm{Ab}$ is an additive bifunctor.
\item[(ET2)] $\mathfrak{s}$ is an additive realization of $\mathbb{E}$.
\item[(ET3)] Let $\delta\in \mathbb{E}(C, A)$ and $\delta'\in \mathbb{E}(C', A')$ be any pair of $\mathbb{E}$-extensions, realized as $\mathfrak{s}(\delta)=[A\mathop{\to}\limits^{x} B\mathop{\to}\limits^{y} C]$ and
$\mathfrak{s}(\delta')=[A'\mathop{\to}\limits^{x'} B'\mathop{\to}\limits^{y'} C']$. For any commutative square
in $\mcC$
\[
\xymatrix{
A\ar[r]^{x}\ar[d]_{a} & B\ar[r]^{y}\ar[d]^{b} & C\\
A'\ar[r]_{x'} & B'\ar[r]_{y'} & C',
}
\]
there exists a morphism $(a, c): \delta\to \delta'$ which is realized by $(a, b, c)$.
\item[(ET3)$^{op}$] Let $\delta\in \mathbb{E}(C, A)$ and $\delta'\in \mathbb{E}(C', A')$ be any pair of 
$\mathbb{E}$-extensions, realized by $A\mathop{\to}\limits^{x} B\mathop{\to}\limits^{y} C$ and 
$A'\mathop{\to}\limits^{x'} B'\mathop{\to}\limits^{y'} C'$, respectively. For any commutative square in $\mcC$
\[
\xymatrix{
A\ar[r]^{x} & B\ar[r]^{y}\ar[d]_{b} & C\ar[d]^{c}\\
A'\ar[r]_{x'} & B'\ar[r]_{y'} & C'
}
\]
there exists a morphism $(a, c): \delta\to \delta'$ which is realized by $(a, b, c)$.

\item[(ET4)] Let $(A, \delta, D)$ and $(B, \delta', F)$ be $\mathbb{E}$-extensions, respectively realized by
$A\mathop{\to}\limits^{f} B\mathop{\to}\limits^{f'} D$ and $B\mathop{\to}\limits^{g} C\mathop{\to}\limits^{g'}
F$. Then there exist an object $E\in \mcC$, a commutative diagram
\[
\xymatrix{
A\ar@{=}[d]\ar[r]^{f} & B\ar[r]^{f'}\ar[d]_{g} & D\ar[d]^{d}\\
A\ar[r]_{h} & C\ar[r]_{h'}\ar[d]_{g'} & E\ar[d]^{e}\\
& F\ar@{=}[r] & F
}
\]
in $\mcC$, and an $\mathbb{E}$-extension $\delta''\in \mathbb{E}(E, A)$ realized by $A\mathop{\to}\limits^{h} C
\mathop{\to}\limits^{h'} E$, which satisfy $\mathfrak{s}(f'\delta ')=[D\mathop{\to}\limits^{d} E\mathop{\to}\limits^{e} F]$, $\delta''d=\delta$ and $f\delta''=\delta'e$.

\item[(ET4)$^{op}$] Let $(D, \delta, B)$ and $(F, \delta', C)$ be $\mathbb{E}$-extensions realized by $D\mathop{\to}\limits^{f'} A\mathop{\to}\limits^{f} B$ and
$F\mathop{\to}\limits^{g'} B\mathop{\to}\limits^{g} C$, respectively. Then there
exist an object $E\in \mcC$, a commutative diagram
\[
\xymatrix{
D\ar@{=}[d]\ar[r]^{d} & E\ar[r]^{e}\ar[d]_{h'} & F\ar[d]^{g'}\\
D\ar[r]_{f'} & A\ar[r]_{f}\ar[d]_{h} & B\ar[d]^{g}\\
& C\ar@{=}[r] & C
}
\]
in $\mcC$ and an $\mathbb{E}$-extension $\delta''\in \mathbb{E}(C, E)$ realized by
$E\mathop{\to}\limits^{h'} A\mathop{\to}\limits^{h} C$ which satisfy $\mathfrak{s}(\delta g')=[D\mathop{\to}\limits^{d} E\mathop{\to}\limits^{e} F]$, $\delta'=e\delta''$ and $d\delta=\delta''g$.
\end{enumerate}
If the above conditions hold true, we call $\mathfrak{s}$ an $\mathbb{E}$-triangulation of $\mcC$, and call the triplet
$(\mcC, \mathbb{E}, \mathfrak{s})$ an externally triangulated category, or for short, extriangulated 
category. Sometimes, for the sake of
simplicity, we only write $\mcC$ instead of $(\mcC, \mathbb{E}, \mathfrak{s}).$ 
\end{definition}

\begin{definition}\cite[Def. 2.15]{Nakaoka1}
Let $(\mcC, \mathbb{E}, \mathfrak{s})$ be a triplet satisfying (ET1) and (ET2).
\begin{enumerate}
\item A sequence $A\mathop{\to}\limits^{x} B\mathop{\to}\limits^{y} C$ is called a conflation if it
realizes some $\mathbb{E}$-extension $\delta\in \mathbb{E}(C, A)$.
\item A morphism $f\in\mcC(A, B)$ is called an inflation if it admits some conflation 
$A\mathop{\to}\limits^{f} B\to C$.
\item A morphism $f\in\mcC(A,B)$ is called a deflation if it admits some conflation $K\to A\mathop{\to}\limits^{f} B$.
\end{enumerate}
\end{definition}

\begin{definition}\cite[Def. 2.17]{Nakaoka1} Let $\mcC$ be an extriangulated category. A subcategory 
$\mcD\subseteq 
\mcC$ is \emph{closed under extensions} if, for any conflation $A\to B\to C$ with $A, C\in \mcD$, we have $B\in \mcD$.
\end{definition}

\begin{definition}\cite[Def. 2.19]{Nakaoka1}
Let $(\mcC, \mathbb{E}, \mathfrak{s})$ be a triplet satisfying (ET1) and (ET2).
\begin{enumerate}
\item If $A\mathop{\to}\limits^{x} B\mathop{\to}\limits^{y} C$ realizes $\delta\in \mathbb{E}(C, A)$, we call the pair $(A\mathop{\to}\limits^{x} B\mathop{\to}\limits^{y} C, \delta)$ an $\mathbb{E}$-triangle,
and we write it in the following way 
$$
A\mathop{\to}\limits^{x} B\mathop{\to}\limits^{y} C\mathop{\dashrightarrow}^{\delta}
$$
\item Let $A\mathop{\to}\limits^{x} B\mathop{\to}\limits^{y} C\mathop{\dashrightarrow}\limits^{\delta}$ and
$A'\mathop{\to}\limits^{x'} B'\mathop{\to}\limits^{y'} C'\mathop{\dashrightarrow}\limits^{\delta'}$ be  $\mathbb{E}$-triangles. If a triplet $(a, b, c)$ realizes $(a, c): \delta\to \delta'$ as in (1), then we write
it as
\[
\xymatrix{
A\ar[r]^{x}\ar[d]_{a} & B\ar[r]^{y}\ar[d]^{b} & C\ar@{-->}[r]^{\delta}\ar[d]^{c} &\\
A'\ar[r]_{x'} & B'\ar[r]_{y'} & C'\ar[r]_{\delta'} &
}
\]
and we call $(a, b, c)$ a morphism of $\mathbb{E}$-triangles.
\end{enumerate}
\end{definition}

\begin{definition}\cite[Def. 3.1]{Nakaoka1}
Let $\mathbb{E}:\mcC^{op}\times\mcC\to Ab$ be an additive bifunctor.
By Yoneda's lemma, any $\mathbb{E}$-extension
$\delta\in \mathbb{E}(C, A)$ induces natural
transformations $\delta_{\#}: \mcC(-,C)\rightarrow \mathbb{E}(-,A)$ and 
$\delta^{\#}: \mcC(A, -)\rightarrow \mathbb{E}(C,-)$. For any $X\in \mcC$, these $(\delta_{\#})_{X}$ and $\delta^{\#}_{X}$ are given as follows
\begin{enumerate}
\item $(\delta_{\#})_{X}: \mcC(X, C)\to 
\mathbb{E}(X, A); f\mapsto \delta f$.
\item $\delta^{\#}_{X}: \mcC(A, X)\to \mathbb{E}(C, X); g\mapsto g\delta$.
\end{enumerate}
We abbreviately denote $(\delta_{\#})_{X}(f)$
and $\delta^{\#}_{X}(g)$ by $\delta_{\#}f$ and
$\delta^{\#}g$, respectively.
\end{definition}

\begin{corollary}\cite[Cor. 3.12]{Nakaoka1} 
Let $\mcC$ be an extriangulated category. For any 
$\mathbb{E}$-triangle $A\mathop{\to}\limits^{x} B\mathop{\to}\limits^{y} C \mathop{\dashrightarrow}\limits^{\delta}$, we have the 
following exact sequences of additive functors
$$\mcC(C,-)\mathop{\longrightarrow}\limits^{\mcC(y,-)} \mcC(B,-)
\mathop{\longrightarrow}\limits^{\mcC(x,-)} \mcC(A,-)\mathop{\longrightarrow}\limits^{\delta^{\#}} \mathbb{E}(C,-)
\mathop{\longrightarrow}\limits^{\mathbb{E}(y,-)} \mathbb{E}(B,-)
\mathop{\longrightarrow}\limits^{\mathbb{E}(x,-)} \mathbb{E}(A,-),$$
$$\mcC(-,A)\mathop{\longrightarrow}\limits^{\mcC(-,x)} \mcC(-,B)
\mathop{\longrightarrow}\limits^{\mcC(-,y)} \mcC(-,C)\mathop{\longrightarrow}\limits^{\delta_{\#}} \mathbb{E}(-,A)
\mathop{\longrightarrow}\limits^{\mathbb{E}(-,x)} \mathbb{E}(-,B)
\mathop{\longrightarrow}\limits^{\mathbb{E}(-,y)} \mathbb{E}(-,C).$$
\end{corollary}

\begin{proposition}\cite[Prop. 3.15]{Nakaoka1}\label{Nakaoka 3.15}
Let $\mcC$ be an 
extriangulated category. The following (and its dual)
holds true.

Let $C\in \mcC$ and $A_1\mathop{\to}\limits^{x_1} B_{1}\mathop{\to}\limits^{y_1} C\mathop{\dashrightarrow}\limits^{\delta_1}$, 
$A_2\mathop{\to}\limits^{x_2} B_{2}\mathop{\to}\limits^{y_2} C\mathop{\dashrightarrow}\limits^{\delta_2}$ be any pair of
$\mathbb{E}$-triangles. Then there exist an object $M\in \mcC$ and a commutative diagram in $\mcC$
\[ 
\xymatrix{
& A_2\ar@{=}[r]\ar[d]_{m_2} & A_2\ar[d]^{x_2}\\
A_1\ar[r]^{m_1}\ar@{=}[d] & M\ar[r]^{e_1}\ar[d]_{e_2} & B_2\ar[d]^{y_2}\\
A_1\ar[r]_{x_1} & B_1\ar[r]_{y_1} & C
}
\]
 which satisfy $\mathfrak{s}(\delta_1 y_2)=[A_1\mathop{\to}\limits^{m_1} M\mathop{\to}\limits^{e_1} B_2]$, $\mathfrak{s}(\delta_2 y_1)=[A_2\mathop{\to}\limits^{m_2} M\mathop{\to}\limits^{e_2} B_1]$ and $m_1\delta_1+m_2\delta_2=0$. 
\end{proposition}

\subsection*{(Co)Resolution dimensions in extriangulated categories}
We recall some definitions in Auslander-Buchweitz theory in
extriangulated categories. These concepts originally appeared in \cite{ABtheory} and have been
adapted to several contexts as in \cite{MendozaSaenzVargasSouto2, MDZtheoryAB}, 
for instance.

\begin{definition} Let $\mcC$ be an extriangulated category. An {\rm  acyclic  $\mathbb{E}$-triangle sequence} is a pair $(C_{\bullet}, Z_\bullet(C_{\bullet})),$ where $C_{\bullet}$ is a sequence 
$\cdots \to C_{n+1}\mathop{\to}\limits^{d_{n+1}} C_{n}\mathop{\to}\limits^{d_{n}} C_{n-1}\to \cdots$
in $\mcC$ and $Z_\bullet(C_{\bullet})$ is a family of $\mathbb{E}$-triangles 
$\{Z_{n+1}(C_{\bullet})\mathop{\to}\limits^{f_{n}} C_{n}
\mathop{\to}\limits^{g_{n}} Z_{n}(C_{\bullet})\dashrightarrow\}_{n\in\mathbb{Z}}$
satisfying that  $f_{n}g_{n+1}=d_{n+1}$ $\forall\,n\in\mathbb{Z}.$ Notice that $C_{\bullet}$ is a chain complex and each $Z_{n}(C_{\bullet})$  is called the {\rm $n$-th $\mathbb{E}$-cycle} of $C_\bullet$ in $\mcC.$ For the sake of simplicity, an acyclic  $\mathbb{E}$-triangle sequence $(C_{\bullet}, Z_\bullet(C_{\bullet}))$ will be denoted by $C_{\bullet}.$
\end{definition}

Let $\mcX\subseteq \mcC$ and $C\in \mcC$.
An \emph{$\mcX$-resolution of $C$} is an acyclic  $\mathbb{E}$-triangle sequence of the form
$
 \cdots \to X_k \to \cdots \to X_1 \to X_0 \to C\to 0
$
where $X_k\in \mcX,$ for every $k\geq 0.$ If $X_k=0$ for every $k>n$, we shall say that
the previous sequence is a \emph{finite 
$\mcX$-resolution of $C$ of length $n$}.
The \emph{resolution dimension of $C$ with respect to $\mcX$} (or the \emph{$\mcX$-resolution dimension of $C$}, for short), denoted $\resdim_{\mcX}(C)$, is the smallest $n \geq 0$ such that $C$ admits a finite $\mcX$-resolution of length $n$. If such $n$ does not exist, we set $\resdim_{\mcX}(C) := \infty$. Dually, an \emph{$\mcX$-coresolution of $C$} is an acyclic  $\mathbb{E}$-triangle sequence of the form $0\to C\to X_0\to X_{-1}\to\cdots\to X_{-k}\to \cdots$ where $X_{-k}\in \mcX,$ for every $k\geq 0.$ If $X_{-k}=0$ for every $k>n$, we shall say that
the previous sequence is a \emph{finite 
$\mcX$-coresolution of $C$ of length $n$}.The smallest $n \geq 0$ such that $C$ admits a finite $\mcX$-coresolution of length $n$ is denoted by $\coresdim_{\mcX}(C)$ and called the \emph{$\mcX$-coresolution dimension} of $C.$ Related with the two above homological dimensions, there exist the following classes of objects in $\mcC$:
\begin{align*}
\mcX^\wedge_n & := \{ C \in \mcC \text{ : } \resdim_{\mcX}(C) \leq n \}, & & \text{and} & \mcX^\wedge & := \bigcup_{n \geq 0} \mcX^\wedge_n, \\
\mcX^\vee_n & := \{ C \in \mcC \text{ : } \coresdim_{\mcX}(C) \leq n \}, & & \text{and} & \mcX^\vee & := \bigcup_{n \geq 0} \mcX^\vee_n.
\end{align*}

Given a class $\mcY$ of objects in $\mcC$,
the \emph{$\mcX$-resolution dimension of 
$\mcY$} is defined as
$$\resdim_{\mcX}(\mcY):=\sup\{\resdim_{\mcX}(Y) : Y\in \mcY\}.$$
The \emph{$\mcX$-coresolution dimension of 
$\mcY$} is defined similarly and denoted by 
$\coresdim_{\mcX}(\mcY)$.


\subsection*{Higher extensions} Let $\mcC$ be an extriangulated category. Following \cite{Nakaoka1}, we recall that an object $P\in \mcC$ is
$\mbE$-projective if for any $\mathbb{E}$-triangle $A\mathop{\to}\limits^{x} B
\mathop{\to}\limits^{y} C\mathop{\dashrightarrow}\limits^{\delta}$ the map
$$\mcC(P,y): \mcC(P,B)\longrightarrow \mcC(P,C)$$
is surjective. We denote by $\mathcal{P}_{\mbE}(\mcC)$ the class of $\mbE$-projective objects in $\mcC$. Dually, the class of $\mbE$-injective objects in $\mcC$ is denoted by $\mcI_{\mbE}(\mcC)$. We say that $\mcC$ has enough $\mbE$-projective objects if for any object $C\in \mcC$, there exists an 
$\mathbb{E}$-triangle $A\to P\to C\dashrightarrow$ satisfying $P\in \mcP_{\mbE}(\mcC)$. Dually, we can
define that $\mcC$ has enough $\mbE$-injective objects.\\

Let $\mcC$ be an extriangulated category and $\mcX, \mcY\subseteq \mcC$. Following \cite[Def. 4.2]{Nakaoka1}, we recall that:

\begin{enumerate}
\item[$\bullet$] $C\in \mcC$ belongs to $\mathrm{Cone}(\mcX, \mcY)$ if $C$ admits a 
conflation $X\to Y\to C$ with $X\in \mcX, Y\in \mcY$.

\item[$\bullet$] $C\in \mcC$ belongs to $\mathrm{CoCone}(\mcX, \mcY)$ if $C$ admits a 
conflation $C\to X\to Y$ with $X\in \mcX, Y\in \mcY$.

\item[$\bullet$] $\mcX$ is \emph{closed under cones} if
$\mathrm{Cone}(\mcX, \mcX)\subseteq \mcX$.
Dually, $\mcX$ is \emph{closed
under cocones} if $\mathrm{CoCone}(\mcX, \mcX)\subseteq \mcX$.
\end{enumerate}

We set $\Omega \mcX:=\mathrm{CoCone}(\mcP_{\mbE}(\mcC), \mcX)$, that is, $\Omega \mcX$ is
the subclass of $\mcC$ consisting of the objects $\Omega X$ admitting an $\mathbb{E}$-triangle
$\Omega X\to P\to X\dashrightarrow$
with $P\in\mcP_{\mbE}(\mcC)$ and $X\in \mcX$. We call $\Omega \mcX$ the syzygy of $\mcX$. Dually, 
the cosyzygy of $\mcX$ is $\Sigma\mcX:=\mathrm{Cone}(\mcX, \mcI_{\mbE}(\mcC))$. We set $\Omega^{0}\mcX:=\mcX$, and define $\Omega^{k}\mcX$ for $k>0$
inductively by
$\Omega^{k}\mcX:=\Omega(\Omega^{k-1}\mcX)=\mathrm{CoCone}(\mcP_{\mbE}(\mcC), \Omega^{k-1}\mcX)$.
We call $\Omega^{k}\mcX$ the $k$-th syzygy of $\mcX$. Dually, the $k$-th cosyzygy $\Sigma^{k}\mcX:=\mathrm{Cone}(\Sigma^{k-1}\mcX, \mcI_{\mbE}(\mcC))$ for $k>0$ and  $\Sigma^{0}\mcX:=\mcX$ 
(see \cite[Def. 4.2 \& Prop. 4.3]{LNheartsoftwin},
for more details).

Let $\mcC$ be an extriangulated category with enough $\mbE$-projectives and $\mbE$-injectives.
In \cite{LNheartsoftwin}, it is shown that $\mathbb{E}(X, \Sigma^{k}Y)\cong \mathbb{E}(\Omega^{k}X, Y)$ 
 for $k\geq 0.$ Thus, the higher extensions groups are defined as 
 $\mathbb{E}^{k+1}(X, Y):=\mathbb{E}(X, \Sigma^{k}Y)\cong \mathbb{E}(\Omega^{k}X, Y),$ for $k\geq 0$. 
 Moreover, the following result is also proved.

\begin{lemma}\cite[Prop. 5.2]{LNheartsoftwin} \label{longExSeq}
Let $\mcC$ be an extriangulated category with enough $\mbE$-projectives and $\mbE$-injectives and $A\to B\to C\dashrightarrow$ be an $\mathbb{E}$-triangle in $\mcC$. Then, for any object $X\in \mcC$ and $k\geq 1$, we have the following
exact sequences 
\[(1)\;
\cdots \to \mathbb{E}^{k}(X, A)\to \mathbb{E}^{k}(X, B)\to \mathbb{E}^{k}(X, C)\to \mathbb{E}^{k+1}(X, A)\to \mathbb{E}^{k+1}(X, B)\to \cdots,
\]
\[(2)\;
\cdots \to \mathbb{E}^{k}(C, X)\to \mathbb{E}^{k}(B, X)\to \mathbb{E}^{k}(A, X)\to \mathbb{E}^{k+1}(C, X)\to \mathbb{E}^{k+1}(B, X)\to \cdots
\]
of abelian groups.
\end{lemma}

\begin{remark}\label{RkAbTri} Let $\mcC$ be an extriangulated category.
 In the case that $\mcC$ is either abelian or triangulated, we do not need  enough $\mbE$-projectives and $\mbE$-injectives to have higher extensions groups. This is because, in that case, we already have defined (for free)  these extension groups 
$\mbE^i(-,-) \colon \mcC^{\rm op} \times \mcC \to  Ab$, with $i \geq 1.$ Indeed, in the abelian case, take the Yoneda's functor $\Ext^i_{\mcC}(-,-)$ (see \cite{Sieg} ),  and in the triangulated case take 
$\Hom_{\mcC}(-,-[i]),$ where $[1]:\mcC\to \mcC$ is the shift functor. Moreover, in such a cases, Lemma \ref{longExSeq} holds true.
\end{remark}

\begin{definition} Let $\mcC$ be an extriangulated category. We say that $\mcC$ is an {\bf extriangulated category with higher extension groups} if one of the following two conditions holds true: (i) $\mcC$ is either abelian or triangulated, and (ii) $\mcC$ has enough $\mbE$-projectives and $\mbE$-injectives.
\end{definition}

Let $\mcC$ be an extriangulated category with higher extension groups. We fix the following notation for $\mcX, \mcY\subseteq \mcC$ and $k \geq 1$.
\begin{itemize}
\item $\mathbb{E}^{k}(\mathcal{X, Y})=0$ if 
$\mathbb{E}^{k}(X, Y)=0$ for every $X\in \mcX$ and $Y\in \mcY$. When $\mcX=\{M\}$ or $\mcY=\{N\}$, 
we shall write $\mathbb{E}^{k}(M, \mcY)=0$ and
$\mathbb{E}^{k}(\mcX, N)=0$, respectively.

\item $\mathbb{E}^{\leq k}(\mcX, \mcY)=0$ if
$\mathbb{E}^{j}(\mcX, \mcY)=0$ for every $1\leq j\leq k$.

\item $\mathbb{E}^{\geq k}(\mcX, \mcY)=0$ if 
$\mathbb{E}^{j}(\mcX, \mcY)=0$ for every $j\geq k$.
\end{itemize}

Recall that the \emph{right $k$-th orthogonal complement} and the \emph{right orthogonal complement of $\mcX$} are defined, respectively, by 
\begin{align*}
\mcX^{\perp_k} := \{ N \in \mcC \mbox{ : } \mathbb{E}^k(\mcX,N) = 0 \}\;\text{ and }\; \mcX^\perp := \bigcap_{k \geq 1} \mcX^{\perp_k}=\{N\in \mcC : \mathbb{E}^{\geq 1}(\mcX, N)=0\}.
\end{align*}
 Dually, we have the \emph{left $k$-th} and the \emph{left orthogonal complements ${}^{\perp_k}\mcX$ and ${}^{\perp}\mcX$ of $\mcX$}, respectively. 


\subsection*{Relative homological dimensions}

Let $\mcC$ be an extriangulated category with higher extension groups. Let 
$\mcX, \mcY \subseteq \mcC$ and $C \in \mcC.$ Following \cite{ABtheory, BMPS, MendozaSaenzVargasSouto2, MDZtheoryAB}, the \emph{relative projective dimension of $C$} is $\pd_{\mcX}(C) := \min \{ m \geq 0 {\rm \ : \ } \mathbb{E}^{\geq m+1}(C,\mcX) = 0 \}$ and $\pd_{\mcX}(\mcY) := \sup\{ \pd_{\mcX}(Y) \text{ : } Y \in \mcY \}.$ Dually, we have the \emph{relative injective dimension} $\id_{\mcX}(C)$ of $C$ and $\id_{\mcX}(\mcY).$ Notice that $\pd_{\mcX}(\mcY)=\id_{\mcY}(\mcX).$ 
If $\mcX=\mcC,$ we just write $\pd(C),$ $\pd(\mcY),$ $\id(C)$ and $\id(\mcY),$ for the (absolute) projective and injective dimensions.

\subsection*{Relative (co)generators}
Let $\mcC$ be an extriangulated category and  $\omega, \mcX\subseteq \mcC$. Following 
\cite{ABtheory, BMPS, MendozaSaenzVargasSouto2, MDZtheoryAB}, we say that $\omega$ is a \emph{relative cogenerator in $\mcX$}, if 
$\omega\subseteq \mcX$ and for each object $X\in \mcX$, there
exists an $\mathbb{E}$-triangle $X\to W\to X'\dashrightarrow$ 
with $W\in \omega$ and $X'\in \mcX$. Assume now that $\mcC$ has higher extension groups.Then, the class $\omega$ is \emph{$\mcX$-injective} if $\id_{\mcX}(\omega)=0$. Dually, we have the notions of
\emph{relative generator} and \emph{$\mcX$-projective}.

\subsection*{Thick subcategories}
Let $\mcC$ be an extriangulated category. A subcategory is \emph{left thick} if it is closed under extensions, direct summands in $\mcC$ and cocones. \emph{Right thick subcategories} are defined dually. Finally, a subcategory is \emph{thick} if it is both left and right thick. For $\mcX\subseteq \mcC$, we shall denote by $\Thick(\mcX)$ the smallest thick subcategory of $\mcC$ containing $\mcX$.

\subsection*{Cotorsion pairs}

Let $\mcC$ be an extriangulated category and $\mathcal{U, V}\subseteq \mcC$ be closed under direct 
summands in $\mcC.$ Following \cite[Def. 4.1]{Nakaoka1} we say that the pair $(\mathcal{U, V})$ is a
\emph{cotorsion pair in $\mcC$} if it satisfies the
following conditions.
\begin{itemize}
\item[(CP1)] $\mathbb{E}(\mathcal{U, V})=0$.
\item[(CP2)] For any $C\in \mcC$, there exist conflations
$$V\to U\to C \quad \mbox{ and } \quad C\to V'\to U'$$ 
satisfying $U, U' \in \mathcal{U}$ and 
$V, V' \in \mathcal{V}$.
\end{itemize}
If $(\mathcal{U, V})$ is a cotorsion pair in $\mcC,$ then $\mathcal{U}={}^{\perp_1}\mathcal{V}$ and $\mathcal{V}=\mathcal{U}^{\perp_1}.$ In particular $\mathcal{U}$ and $\mathcal{V}$ are additive subcategories of $\mcC.$

\section{\textbf{Some technical results}}

Although abelian and extriangulated categories share homological tools and constructions, it is not always clear how these mechanisms can be carried to different contexts. In 
this section we begin enunciating results from \cite{HMPcut} which have been adapted to extriangulated categories. For the rest of this work $\mcC,$ denotes an extriangulated
category. The following result is a generalization of \cite[Lem. 3.3]{HMPcut}.

\begin{lemma}\label{lem:technical_lemma_1}
Let $\omega, \mcS \subseteq \mcC$ with $\omega$ closed under extensions and $\omega \cap \mcS$ be a relative generator in $\omega$. Let $C \in \mcC$ and 
\begin{align}
C_\bullet :\quad \cdots \to W_n\to W_{n-1} \to \cdots \to W_1 \to W_0 \to C\to 0 \label{eqn:resolution_omega}
\end{align}
be an $\omega$-resolution of $C.$ Then, there exist 
$\mathbb{E}$-triangles
\begin{align}
G_j & \to X_{j+1} \to Z_{j+1}(C_\bullet)\dashrightarrow & & \text{and} & X_{j+1} & \to F_j \to X_j\dashrightarrow \label{eqn:technical} 
\end{align}
where $X_0 := C$, $F_j \in \omega \cap \mcS$ and $G_j \in \omega,$ for every $j\geq 0$.  
\end{lemma}

\begin{proof}
Let us prove this result by induction on $j$. 
\begin{itemize}
\item \underline{Initial step}: For the case $j = 0$, since $\omega \cap \mcS$ is a relative generator in $\omega$, there is an
$\mathbb{E}$-triangle $G_0 \to F_0 \to W_0\dashrightarrow$ with $G_0 \in \omega$ and $F_0 \in \omega \cap \mcS$. From condition
(ET4)$^{op}$ we can form the following 
commutative diagram in $\mcC$
\[
\begin{tikzpicture}[description/.style={fill=white,inner sep=2pt}] 
\matrix (m) [matrix of math nodes, row sep=2.5em, column sep=2.5em, text height=1.25ex, text depth=0.25ex] 
{ 
G_0 & X_1 & Z_1(C_\bullet) \\
G_0 & F_{0} & W_{0}  \\
{} & C & C  \\
};  
\path[->] 
(m-1-1) edge (m-1-2) (m-1-2) edge (m-1-3)
(m-2-1) edge (m-2-2) (m-2-2) edge (m-2-3)
(m-1-2) edge (m-2-2) (m-2-2) edge (m-3-2)
(m-1-3) edge (m-2-3) (m-2-3) edge (m-3-3)
; 
\path[-,font=\scriptsize]
(m-3-2) edge [double, thick, double distance=2pt] (m-3-3)
(m-1-1) edge [double, thick, double distance=2pt] (m-2-1)
;
\end{tikzpicture} 
\]
where $G_0\to X_1\to Z_1(C_\bullet)\dashrightarrow$ and $X_1\to F_0\to C\dashrightarrow$ are $\mathbb{E}$-triangles. Hence, the 
existence of such $\mathbb{E}$-triangles holds in this case.

\item \underline{Induction step}: Now suppose that for $1 \leq j$ there are $\mathbb{E}$-triangles  
$G_j \to X_{j+1} \to Z_{j+1}(C_\bullet)\dashrightarrow$ and 
$X_{j+1} \to F_j \to X_j\dashrightarrow$, with $F_j \in \omega \cap \mcS$ and $G_j \in \omega$. From 
Proposition~\ref{Nakaoka 3.15}
and by considering the $\mathbb{E}$-triangle $Z_{j+2}(C_\bullet) \to W_{j+1} \to Z_{j+1}(C_\bullet)\dashrightarrow$ we can construct
a commutative diagram in $\mcC$
\[
\begin{tikzpicture}[description/.style={fill=white,inner sep=2pt}] 
\matrix (m) [matrix of math nodes, row sep=2.5em, column sep=2.5em, text height=1.25ex, text depth=0.25ex] 
{ 
{} & G_j & G_j \\
Z_{j+2}(C_\bullet) & F'_{j+1} & X_{j+1}  \\
Z_{j+2}(C_\bullet) & W_{j+1} & Z_{j+1}(C_\bullet)  \\
};  
\path[->] 
(m-2-1) edge (m-2-2) (m-2-2) edge (m-2-3) 
(m-3-1) edge (m-3-2) (m-3-2) edge (m-3-3) 
(m-1-2) edge (m-2-2) (m-2-2) edge (m-3-2)
(m-1-3) edge (m-2-3) (m-2-3) edge (m-3-3)
; 
\path[-,font=\scriptsize]
(m-1-2) edge [double, thick, double distance=2pt] (m-1-3)
(m-2-1) edge [double, thick, double distance=2pt] (m-3-1)
;
\end{tikzpicture} 
\]
where $G_j\to F'_{j+1}\to W_{j+1}\dashrightarrow$ is an
$\mathbb{E}$-triangle. Since $\omega$ is closed under extensions, $F'_{j+1} \in \omega$. By using again that $\omega \cap \mcS$ is a relative generator in $\omega$, we have an $\mathbb{E}$-triangle
$G_{j+1} \to F_{j+1} \to F'_{j+1}\dashrightarrow$ with $F_{j+1} \in \omega \cap \mcS$ and $G_{j+1} \in \omega$. From condition
(ET4)$^{op}$ again we get a
commutative diagram
\[
\begin{tikzpicture}[description/.style={fill=white,inner sep=2pt}] 
\matrix (m) [matrix of math nodes, row sep=2.5em, column sep=2.5em, text height=1.25ex, text depth=0.25ex] 
{ 
G_{j+1} & X_{j+2} & Z_{j+2}(C_\bullet) \\
G_{j+1} & F_{j+1} & F'_{j+1}  \\
{} & X_{j+1} & X_{j+1}  \\
};  
\path[->] 
(m-1-1) edge (m-1-2) (m-1-2) edge (m-1-3)
(m-2-1) edge (m-2-2) (m-2-2) edge (m-2-3)
(m-1-2) edge (m-2-2) (m-2-2) edge (m-3-2)
(m-1-3) edge (m-2-3) (m-2-3) edge (m-3-3)
; 
\path[-,font=\scriptsize]
(m-3-2) edge [double, thick, double distance=2pt] (m-3-3)
(m-1-1) edge [double, thick, double distance=2pt] (m-2-1)
;
\end{tikzpicture} 
\]
where $G_{j+1}\to X_{j+2}\to Z_{j+2}(C_\bullet)\dashrightarrow$ and $X_{j+2}\to F_{j+1}\to X_{j+1}\dashrightarrow$ are $\mathbb{E}$-triangles. Therefore, the result follows.
\end{itemize}
\end{proof}

The following result is the generalization of \cite[Lem. 3.4]{HMPcut}.

\begin{lemma}\label{lem:technical_lemma_2}
Let $\omega\subseteq \mcC$ be closed under extensions. If $W \to B \to C\dashrightarrow$ is an
$\mathbb{E}$-triangle with $W \in \omega$ and $C \in \omega^\wedge$, then $B \in \omega^\wedge$ and $\resdim_\omega(B) \leq \resdim_\omega(C)$. 
\end{lemma}

\begin{proof}
For $\resdim_{\omega}(C) = 0$, the result follows since $\omega$ is closed under extensions. So we may assume that $\resdim_{\omega}(C) \geq 1$. Then, there exists an $\mathbb{E}$-triangle $W' \to W_0 \to C\dashrightarrow$ with $W_0 \in \omega$ and $\resdim_{\omega}(W') = \resdim_{\omega}(C) - 1$. From 
Proposition~\ref{Nakaoka 3.15}, one can
form the following commutative diagram in $\mcC$
\[
\begin{tikzpicture}[description/.style={fill=white,inner sep=2pt}] 
\matrix (m) [matrix of math nodes, row sep=2.5em, column sep=2.5em, text height=1.25ex, text depth=0.25ex] 
{ 
{} & W & W \\
W' & E & B  \\
W' & W_0 & C  \\
};  
\path[->] 
(m-2-1) edge (m-2-2) (m-2-2) edge (m-2-3) 
(m-3-1) edge (m-3-2) (m-3-2) edge (m-3-3) 
(m-1-2) edge (m-2-2) (m-2-2) edge (m-3-2)
(m-1-3) edge (m-2-3) (m-2-3) edge (m-3-3)
; 
\path[-,font=\scriptsize]
(m-1-2) edge [double, thick, double distance=2pt] (m-1-3)
(m-2-1) edge [double, thick, double distance=2pt] (m-3-1)
;
\end{tikzpicture} 
\]
where $W'\to E\to B\dashrightarrow$ is an
$\mathbb{E}$-triangle and $E\in \omega$ due to $\omega$ is closed under extensions. Therefore, from the middle row we get that $B\in \omega^{\wedge}$. 
\end{proof}

\begin{lemma}\label{lem: w^ cerrada por sumandos} Let $A\to B\to C\dashrightarrow$ be an $\mathbb{E}$-triangle in
$\mcC$.
\begin{enumerate}
\item If $C=C_1\oplus C_2$ in $\mcC,$ then
there exist $\mathbb{E}$-triangles $A_i \to B \to C_i \dashrightarrow$ for every $i=1, 2$.
\item The sequence $C\mathop{\longrightarrow}\limits^{\nabla} C\oplus C\mathop{\longrightarrow}\limits^{\footnotesize{(-1 \, 1)}} C$ realizes the split $\mathbb{E}$-extension $0\in \mathbb{E}(C,C),$ for any $C\in \mcC,$ where
$\nabla:=\footnotesize{\left(\begin{array}{c}
1\\1
\end{array}
\right)}: C\to C\oplus C$ is the codiagonal map.
\end{enumerate}
\end{lemma}

\begin{proof} (1) From condition (ET4)$^{op}$ we have the following diagram 
\[
\begin{tikzpicture}[description/.style={fill=white,inner sep=2pt}] 
\matrix (m) [matrix of math nodes, row sep=2.5em, column sep=2.5em, text height=1.25ex, text depth=0.25ex] 
{ 
A & A_2 & C_1 \\
A & B & C  \\
{} & C_2 & C_2,  \\
};  
\path[->] 
(m-1-1) edge (m-1-2) (m-1-2) edge (m-1-3)
(m-2-1) edge (m-2-2) (m-2-2) edge (m-2-3)
(m-1-2) edge (m-2-2) (m-2-2) edge (m-3-2)
(m-1-3) edge node[right] {\footnotesize$\left( \begin{array}{c}
1\\ 0
\end{array}
\right)$} (m-2-3) (m-2-3) edge node[right] {\footnotesize$\left( \begin{array}{cc}
0 & 1
\end{array}
\right)$} (m-3-3)
; 
\path[-,font=\scriptsize]
(m-3-2) edge [double, thick, double distance=2pt] (m-3-3)
(m-1-1) edge [double, thick, double distance=2pt] (m-2-1)
;
\end{tikzpicture} 
\]
where $A_2\to B\to C_2\dashrightarrow$ is an $\mathbb{E}$-triangle. In a similar way, we obtain an $\mathbb{E}$-triangle
$A_1\to B\to C_1\dashrightarrow$.
\

(2) It follows from \cite[Rmk. 2.11-(2)]{Nakaoka1}.
\end{proof}

The following result generalizes \cite[Lem. 3.5]{HMPcut}.

\begin{lemma}\label{lem:main_lemma}
Let $\mcC$ be with higher extension groups, $\omega, \mcS \subseteq \mcC$ be such that $\omega$ is closed under extensions and isomorphisms, and let $\omega \cap \mcS$ be an $\omega$-projective relative generator in $\omega.$ Then, the following statements hold true.
\begin{enumerate}
\item $\omega \cap \mcS$ is an $\omega^\wedge$-projective relative generator in $\omega^\wedge$. Moreover, for any $C \in \omega^\wedge$ with $\resdim_{\omega}(C) \geq 1$, there exists an $\mathbb{E}$-triangle
$K \to F \to C\dashrightarrow$ such that $F \in \omega \cap \mcS$ and $\resdim_{\omega}(K) = \resdim_{\omega}(C) - 1$. 

\item $\omega^\wedge$ is closed under extensions.

\item If $\omega$ is closed under direct summands, then so is $\omega^\wedge$. 

\item If $\mcS$ is closed under cones and cocones, then $\omega^\wedge \cap \mcS = (\omega \cap \mcS)^\wedge$.
\end{enumerate}
\end{lemma}

\begin{proof} \
\begin{enumerate}
\item From \cite[Lem. 3.8]{MDZtheoryAB} and its dual, we have that $\pd_{\omega^\wedge}(\omega \cap \mcS) = \id_{\omega \cap \mcS}(\omega^\wedge) = \id_{\omega \cap \mcS}(\omega) = \pd_{\omega}(\omega \cap \mcS) = 0$, and so $\omega \cap \mcS$ is $\omega^\wedge$-projective. Now, we show that $\omega \cap \mcS$ is a relative generator in $\omega^\wedge$. We use induction on $n := \resdim_{\omega}(C),$ for $C \in \omega^\wedge$. 
\begin{itemize}
\item \underline{Initial step}: This is clear since $\omega$ is closed under isomorphisms. 

\item \underline{Induction step}: For the case $n \geq 1$, we have an $\omega$-resolution of $C$
\[
0\to W_n \to W_{n-1} \to \cdots \to W_1 \to W_0 \to C\to 0
\]
with $W_k \in \omega$ for every $0 \leq k \leq n$. By Lemma \ref{lem:technical_lemma_1} and the notation therein, we have that $X_n \in \omega$ since $G_{n-1}, Z_n(C_{\bullet})=W_n \in \omega$ and $\omega$ is closed under extensions. Glueing together the $\mathbb{E}$-triangles in Lemma \ref{lem:technical_lemma_1}~\eqref{eqn:technical}, we get an acyclic $\mbE$-triangle sequence 
\[
0\to X_n \to F_{n-1} \to \cdots \to F_1 \to F_0 \to C\to 0
\]
with $F_k \in \omega \cap \mcS,$ for every $0 \leq k \leq n-1$. Thus, by 
considering the $\mathbb{E}$-triangle $X_1 \to F_0 \to C\dashrightarrow$ where $F_0 \in \omega \cap \mcS$ and $X_1 \in \omega^\wedge$ with $\resdim_{\omega}(X_1) = n - 1,$
we can conclude that $\omega \cap \mcS$ is a relative generator in $\omega^\wedge$. 
\end{itemize}

\item Suppose we are given an $\mathbb{E}$-triangle $A \to B \to C\dashrightarrow$ with $A, C \in \omega^\wedge$. Let us use induction on $n := \resdim_{\omega}(A)$ to show that $B \in \omega^\wedge$. 
\begin{itemize}
\item \underline{Initial step}: If $\resdim_{\omega}(A) = 0$, the result follows by Lemma \ref{lem:technical_lemma_2} and the hypothesis that $\omega$ is closed under isomorphisms.

\item \underline{Induction step}: We may assume that $\resdim_{\omega}(A) \geq 1$. Suppose that for any 
$\mathbb{E}$-triangle $A' \to B' \to C'\dashrightarrow$ with $C' \in \omega^\wedge$ and $\resdim_{\omega}(A') \leq n-1$, one has that $B' \in \omega^\wedge$. 

On the one hand, by part (1), there is an $\mathbb{E}$-triangle $C' \to P \to C
\dashrightarrow$ with $P \in \omega \cap \mcS$ and $C' \in \omega^\wedge$. Thus, by Proposition~\ref{Nakaoka 3.15} 
we can form the following commutative diagram
\[
\begin{tikzpicture}[description/.style={fill=white,inner sep=2pt}] 
\matrix (m) [matrix of math nodes, row sep=2.5em, column sep=2.5em, text height=1.25ex, text depth=0.25ex] 
{ 
{} & C' & C' \\
A & E & P  \\
A & B & C,  \\
};  
\path[->] 
(m-2-1) edge (m-2-2) (m-2-2) edge (m-2-3) 
(m-3-1) edge (m-3-2) (m-3-2) edge (m-3-3) 
(m-1-2) edge (m-2-2) (m-2-2) edge (m-3-2)
(m-1-3) edge (m-2-3) (m-2-3) edge (m-3-3)
; 
\path[-,font=\scriptsize]
(m-1-2) edge [double, thick, double distance=2pt] (m-1-3)
(m-2-1) edge [double, thick, double distance=2pt] (m-3-1)
;
\end{tikzpicture} 
\]
where $A\to E\to P\dashrightarrow$ is an
$\mathbb{E}$-triangle.
Since $\pd_{\omega^\wedge}(\omega \cap \mcS) = 0$, the conflation $A\to E\to P$ realizes the
split $\mathbb{E}$-extension, and then $E \simeq A \oplus P$ \cite[Rmk. 2.11-(1)]{Nakaoka1}. 

On the other hand, by part (1) again, 
there is an $\mathbb{E}$-triangle $A' \to W_0 \to A\dashrightarrow$ with $W_0 \in \omega \cap \mcS$ and $\resdim_{\omega}(A') = \resdim_{\omega}(A) - 1$. Thus, 
by considering the direct sum of this $\mathbb{E}$-triangle and $0 \to P \mathop{\to}\limits^{1} P\dashrightarrow$ we get an $\mathbb{E}$-triangle
$A' \to W_0 \oplus P \to A \oplus P\dashrightarrow$ with $W_0 \oplus P \in \omega$ since $\omega$ is 
closed under extensions.

Now, applying (ET4)$^{op}$ to the $\mathbb{E}$-triangles $A' \to W_0 \oplus P \to A \oplus P\dashrightarrow$ and
$C'\to A \oplus P\to B\dashrightarrow,$ we have the following 
commutative diagram in $\mcC$
\[
\begin{tikzpicture}[description/.style={fill=white,inner sep=2pt}] 
\matrix (m) [matrix of math nodes, row sep=2.5em, column sep=2.5em, text height=1.25ex, text depth=0.25ex] 
{ 
A' & E' & C' \\
A' & W_{0}\oplus P & A\oplus P  \\
{} & B & B,  \\
};  
\path[->] 
(m-1-1) edge (m-1-2) (m-1-2) edge (m-1-3)
(m-2-1) edge (m-2-2) (m-2-2) edge (m-2-3)
(m-1-2) edge (m-2-2) (m-2-2) edge (m-3-2)
(m-1-3) edge (m-2-3) (m-2-3) edge (m-3-3)
; 
\path[-,font=\scriptsize]
(m-3-2) edge [double, thick, double distance=2pt] (m-3-3)
(m-1-1) edge [double, thick, double distance=2pt] (m-2-1)
;
\end{tikzpicture} 
\]
where $A'\to E'\to C'\dashrightarrow$ and $E'\to W_0\oplus P\to B\dashrightarrow$ are $\mathbb{E}$-triangles.
By using induction hypothesis in the second column of the previous diagram
the result follows. 
\end{itemize}

\item  Given an object $C = C_1 \oplus C_2 \in \omega^\wedge$ with $n := \resdim_{\omega}(C)$,
we use induction on $n$ to prove that $C_1, C_2 \in \omega^\wedge$.
\begin{itemize}
\item \underline{Initial step}: The case $n = 0$ is trivial. 

\item \underline{Induction step}: Suppose that every direct summand in $\mcC$ of an object $D \in \omega^\wedge$ with $\resdim_{\omega}(D) \leq n-1$ belongs to $\omega^\wedge$. By part (1), there exists an $\mathbb{E}$-triangle 
\[
K \to W \xrightarrow{{\scriptsize \left(\begin{array}{c} h_1 \\ h_2 \end{array}\right)}} C \mathop{\dashrightarrow}^{\delta}
\]
with $W \in \omega \cap \mcS$ and $\resdim_{\omega}(K) = \resdim_{\omega}(C) - 1$. Thus, for $i = 1, 2$ we have 
$\mathbb{E}$-triangles $K_i \to W \xrightarrow{h_i} C_i \mathop{\dashrightarrow}\limits^{\delta_{i}}$ from
Lemma~\ref{lem: w^ cerrada por sumandos}. 

By taking the direct sum of $\delta_1$ and $\delta_2$ we get the following $\mathbb{E}$-triangle 
\[
K_1 \oplus K_2 \to W \oplus W \xrightarrow{{\scriptsize \left( \begin{array}{cc} h_1 & 0 \\ 0 & h_2 \end{array} \right)}} C \mathop{\dashrightarrow}\limits^{\delta_1 \oplus \delta_2},
\]
where $W \oplus W \in \omega$ since $\omega$ is closed under extensions. Then, by Lemma~\ref{lem: w^ cerrada por sumandos} and \cite[Prop. 3.17]{Nakaoka1}
we can form the following commutative diagram
\[
\begin{tikzpicture}[description/.style={fill=white,inner sep=2pt}]
\matrix (m) [matrix of math nodes, row sep=2em, column sep=2em, text height=1.25ex, text depth=0.25ex]
{ 
 K & W & C \\
 K_1 \oplus K_2 & W \oplus W & C \\
 W & W & {} \\
};
\path[->]
(m-1-1) edge (m-1-2) (m-1-2) edge (m-1-3)
(m-2-1) edge (m-2-2) (m-2-2) edge (m-2-3)
(m-1-2) edge node[right] {\footnotesize$\nabla$} (m-2-2)
(m-2-2) edge node[right] {\footnotesize$(-1\, 1)$} (m-3-2)
(m-1-1) edge (m-2-1) (m-2-1) edge (m-3-1)
;
\path[-,font=\scriptsize]
(m-3-1) edge [double, thick, double distance=2pt] (m-3-2)
(m-1-3) edge [double, thick, double distance=2pt] (m-2-3)
;
\end{tikzpicture}
\]
The first column gives rise to a conflation $K \to K_1 \oplus K_2 \to W $ which splits 
since $W \in \omega \cap \mcS$, $K \in \omega^\wedge$ and $\pd_{\omega^\wedge}(\omega \cap \mcS) = 0$. Thus, $K_1 \oplus K_2 \simeq K \oplus W$ and 
\[
\hspace{1.5cm} \resdim_{\omega}(K_1 \oplus K_2) = \resdim_{\omega}(K \oplus W) \leq \resdim_{\omega}(K) \leq n-1.
\]
By induction hypothesis, we have that $K_i \in \omega^\wedge$. Hence, from the $\mbE$-triangle $K_i \to W \xrightarrow{h_i} C_i \mathop{\dashrightarrow}\limits^{\delta_{i}}$ we finally have that $C_i \in \omega^\wedge$. 
\end{itemize}

\item The proof is similar to the given one in \cite[Lem. 3.5-(4)]{HMPcut}.
\end{enumerate}
\end{proof}

\begin{lemma}\label{lem:relative_equality}
Let $\mcC$ be with higher extension groups and $\mcX, \omega, \mcS \subseteq \mcC$ satisfying the following conditions:
\begin{enumerate}
\item $\omega$ is closed under extensions and isomorphisms; 

\item $\omega \cap \mcS$ is closed under direct summands and a relative generator in $\omega$;

\item $\omega \cap \mcS \subseteq \mcX \cap \mcS$;

\item $\mcX \cap \mcS$ is closed under cocones;  

\item $\id_{\mcX \cap \mcS}(\omega \cap \mcS) = 0$.
\end{enumerate} 
Then, $\omega \cap \mcS = \mcX \cap \omega^\wedge \cap \mcS$. 
\end{lemma}

\begin{proof}
The proof given in \cite[Lem. 4.1]{HMPcut} can be easily extended to the context of extriangulated 
categories.
\end{proof}


\section{\textbf{Cut Frobenius pairs vs. Cut AB-contexts}}

Let us commence with the analogs of left Frobenius pairs and weak Auslander-Buchweitz contexts
in extriangulated categories. These notions have been given previously, for example in \cite{BMPS, TGone, ma2021new}, however in this paper
we work through of a relativization that depends on a class of objects in an extriangulated category $\mcC$ with higher extension groups.
Following \cite{TGone}, we first recall the notion of left Frobenius pair in such extriangulated categories.

\begin{definition}\cite[Def. 4.7]{TGone}\label{def: left Frob}
Let $\mcC$ be an extriangulated category with higher extension groups and  $\mcX, \omega\subseteq \mcC$. We say that $(\mcX, \omega)$ is 
a \textbf{left Frobenius pair} in $\mcC$ if the following 
conditions are satisfied.
\begin{itemize}
\item[(lF1)] $\mcX$ is left thick.
\item[(lF2)] $\omega$ is closed under direct summands in $\mcC.$
\item[(lF3)] $\omega$ is an $\mcX$-injective relative cogenerator in $\mcX$.
\end{itemize}
If the dual conditions hold true for the pair $(\nu, \mcY)$; that is, (1) $\mcY$ is right thick; (2) $\nu$ is closed under direct summands in $\mcC;$ and (3) $\nu$ is a $\mcY$-projective 
relative generator in $\mcY$, we shall say that the pair $(\nu, \mcY)$ is a \textbf{right Frobenius pair in $\mcC$}.
\end{definition}

\begin{remark}\label{rmk: frob1}
For any left Frobenius pair $(\mcX, \omega)$ in $\mcC$, we have
that $\mcX^{\wedge}$ is a thick subcategory of $\mcC$ and $\id_{\mcX}(\omega^{\wedge})=0$ (see \cite[Props. 3.10, 3.14 \& 3.15 and Lem. 3.8]{MDZtheoryAB} or
\cite[Prop. 3.8 \& Lem. 3.10]{ma2021new}).
\end{remark}

\begin{example}\label{ex: GP, P} 
Let $\xi$ be a proper class of $\mathbb{E}$-triangles in $\mcC$ \cite[Def. 3.1]{HZZgorensteinness}. It is known that if $\mcC$ has
enough $\xi$-projectives and enough $\xi$-injectives, then $(\mcC, \mathbb{E}_{\xi}, \mathfrak{s}_{\xi})$ is an extriangulated category where $\mathfrak{s}_{\xi}:=\mathfrak{s}|_{\mathbb{E}_{\xi}}$ and 
$\mathbb{E}_{\xi}:=\mathbb{E}|_{\xi}$ \cite[Def. 4.1 \& Thm. 3.2]{HZZgorensteinness}. Thus, from \cite[Ex. 4.8]{TGone} we have that $(\mathcal{GP}(\xi), \mathcal{P}(\xi))$ is a 
left Frobenius pair in $(\mcC, \mathbb{E}_{\xi}, \mathfrak{s}_{\xi})$\footnote{Definition~\ref{def: left Frob} coincides with \cite[Def. 4.7]{TGone}
by considering $\xi$ as the class of all the $\mathbb{E}$-triangles in $\mcC$.} where
$\mathcal{GP}(\xi)$ is the class of $\xi$-$\mathcal{G}$projective objects and $\mathcal{P}(\xi)$ denotes the 
class of $\xi$-projective objects \cite[Defs. 4.8 \& 4.1]{HZZgorensteinness}. In particular, 
$\id_{\mathcal{GP}(\xi)}(\mathcal{P}(\xi)^{\wedge})=0$ \cite[Sections 4.1 \& 5.1]{HZZgorensteinness}.
\end{example}

Now, we extend the notion of cut Frobenius pair, given in \cite[Def. 3.6]{HMPcut}, for extriangulated
categories with higher extension groups. 

\begin{definition}\label{def: cut left Frobenius pair}
Let $\mcC$ be an extriangulated category with higher extension groups and $\mcX, \omega, \mcS\subseteq \mcC$. We say that $(\mcX, \omega)$ is a \textbf{left Frobenius pair cut on $\mcS$} if the
following conditions are satisfied.
\begin{enumerate}
\item $\mcX$ is left thick.
\item $(\mcX\cap \mcS, \omega\cap\mcS)$ is a left Frobenius pair in $\mcC.$
\item $\omega\cap \mcS$ is $\omega$-projective and a relative generator in $\omega.$
\item $\omega$ is closed under extensions and direct summands in $\mcC.$
\end{enumerate}

In case that, the dual conditions are satisfied for the pair $(\nu, \mcY)$ we shall say that it is a 
\textbf{right Frobenius pair cut on $\mcS$}.
\end{definition}

\begin{remark}\label{rmk: frob}
If $(\mcX,\omega)$ is a left Frobenius pair cut on $\mcS$, then
all the statements in Lemma \ref{lem:main_lemma} hold true (these results are also valid for the case $\mathcal{S}:=\mcC$). In particular, $\omega^\wedge$ is closed under extensions and direct summands in $\mcC.$ 
Moreover, since $(\mcX, \omega)$ is a left Frobenius pair cut on $\mcS$, we also have from Lemma~\ref{lem:relative_equality} that
$\omega\cap \mcS=\mcX\cap \mcS\cap \omega^{\wedge}$. Thus, $\omega\cap \mcS$ is closed under extensions since
$\mcX\cap \mcS$ and $\omega^{\wedge}$ are closed under extensions.
\end{remark}

Next, we extend the notion of cut Auslander-Buchweitz context, given in \cite[Def. 3.14]{HMPcut},
for extriangulated categories with higher extension groups.

\begin{definition}\label{def: cut left AB context}
Let $\mcC$ be an extriangulated category with higher extension groups,  $\mcA, \mcB, \mcS\subseteq \mcC$ and $\omega:=\mcA\cap\mcB$.
We say that $(\mcA, \mcB)$ is a \textbf{left weak AB context cut on 
$\mcS$} if the following conditions are satisfied.
\begin{enumerate}
\item $(\mcA, \omega)$ is a left Frobenius pair cut on $\mcS.$
\item $\mcB\cap \mcS$ is right thick.
\item $\mcB\cap \mcS\subseteq (\mcA\cap \mcS)^{\wedge}$.
\end{enumerate}
Dually, the notion of \textbf{right weak AB context cut on $\mcS$} is defined. 
\end{definition}

The following table summarizes all the properties mentioned previously.
The ones marked by $\checkmark$ follow by definition while those ones marked by 
$\star$ are a consequence from Remarks~\ref{rmk: frob1} and~\ref{rmk: frob}.$\,$

\begin{center}
\begin{tabular}{ccccc}
\hline
& & Closedness by  & &\\
& extensions & direct summands & cocones & cones\\
\hline
Left Frobenius $(\mcX, \omega)$ & & & & \\
\hline
$\mcX$ & $\checkmark$ & $\checkmark$ & $\checkmark$ &\\
$\omega$ & $\star$ & $\checkmark$ & &\\
$\mcX^{\wedge}$ & $\star$ & $\star$ & $\star$ & $\star$\\
\hline
Cut left Frobenius $(\mcX, \omega)$ & & & & \\
\hline
$\mcX$ & $\checkmark$ & $\checkmark$ & $\checkmark$ &\\
$\omega$ & $\checkmark$ & $\checkmark$ & &\\
$\mcX\cap \mcS$ & $\checkmark$ & $\checkmark$ & $\checkmark$ &\\
$\omega\cap \mcS$ & $\star$ & $\checkmark$ & &\\
$\omega^{\wedge}$ & $\star$ & $\star$ & &\\
\hline
Cut left weak AB-context $(\mcA, \mcB)$ & & & & \\
\hline
$\mcA$ & $\checkmark$ & $\checkmark$ & $\checkmark$ &\\
$\omega:=\mcA\cap \mcB$ & $\checkmark$ & $\checkmark$ & &\\
$\mcA\cap \mcS$ & $\checkmark$ & $\checkmark$ & $\checkmark$ &\\
$\omega\cap \mcS$ & $\star$ & $\checkmark$ & &\\
$\mcB\cap \mcS$ & $\checkmark$ & $\checkmark$ & & $\checkmark$\\
\hline 
\end{tabular}
\end{center}$\,$\\

In the following lines we show that
there exists a one-to-one correspondence between the class of left Frobenius
pairs cut on $\mcS$, denoted by $\mathfrak{F}_{\mcS}$, and the class of left 
weak AB-contexts on $\mcS$, denoted by $\mathfrak{C}_{\mcS}$. It is worth mentioning that
such correspondence will be given under equivalence relations whose definition we give below.

\begin{definition}\label{def:Frobenius_and_AB_relations}
Let $\mcC$ be an extriangulated category with higher extension groups and let $\mcS \subseteq \mcC$. For $(\mcX,\omega), (\mcX',\omega') \in \mfF_{\mcS}$ and $(\mcA,\mcB)$, $(\mcA',\mcB') \in \mfC_{\mcS}$, we shall say that:
\begin{enumerate}
\item $(\mcX,\omega)$ is \textbf{related} to $(\mcX',\omega')$ in $\mfF_{\mcS}$, denoted $(\mcX,\omega) \sim (\mcX',\omega')$, if $\mcX \cap \mcS = \mcX' \cap \mcS$ and $\omega \cap \mcS = \omega' \cap \mcS$;

\item $(\mcA,\mcB)$ is \textbf{related} to $(\mcA',\mcB')$ in $\mfC_{\mcS}$, denoted $(\mcA,\mcB) \sim (\mcA',\mcB')$, if $\mcA \cap \mcS = \mcA' \cap \mcS$ and $\mcA \cap \mcB \cap \mcS = \mcA' \cap \mcB' \cap \mcS$. 
\end{enumerate} 
\end{definition}

We denote by $[\mcX,\omega]_{\mfF_{\mcS}}$ the equivalence class of $(\mcX,\omega)$ in 
$\mfF_{\mcS} /\!\! \sim$. Similarly, $[\mcA,\mcB]_{\mfC_{\mcS}}$ denotes the equivalence class of $(\mcA,\mcB)$ in $\mfC_{\mcS} /\!\! \sim$. The following result generalizes \cite[Thm. 4.6]{HMPcut} in extriangulated categories.

\begin{theorem}[First correspondence theorem]\label{theo:correspondence_1}
Let $\mcC$ be an extriangulated category with higher extension groups and let $\mcS \subseteq \mcC$ be closed under cones and cocones. Then, the correspondence 
$\Phi_{\mcS}\colon \mfF_{\mcS} /\!\! \sim \mbox{} \to \mfC_{\mcS} /\!\! \sim,\; 
[\mcX,\omega]_{\mfF_{\mcS}} \mapsto [\mcX,\omega^\wedge]_{\mfC_{\mcS}},$ is bijective and the inverse 
$\Psi_{\mcS}$ is given by  
$[\mcA,\mcB]_{\mfC_{\mcS}} \mapsto [\mcA,\mcA \cap \mcB]_{\mfF_{\mcS}}$.
\end{theorem}

\begin{proof}
First, we show that the mappings $\Phi_{\mcS}$ and $\Psi_{\mcS}$ are well-defined. 

On the one hand,  by 
Definition~\ref{def: cut left AB context}, we have that $(\mcA,\mcA \cap \mcB) \in \mfF_{\mcS}$ for every $(\mcA,\mcB) \in \mfC_{\mcS}$. Moreover, it is clear that $\Psi_{\mcS}([\mcA,\mcB]_{\mfC_{\mcS}})$ does not depend on the chosen representative $(\mcA,\mcB) \in \mfC_{\mcS}$. On the other hand,
$\Phi_{\mcS}$ does not depend on representatives by Lemma \ref{lem:relative_equality}. So, we only need to prove that 
if $(\mcX,\omega) \in \mfF_{\mcS}$ then $(\mcX,\omega^\wedge) \in \mfC_{\mcS}$. For this, we see that the conditions in Definition~\ref{def: cut left AB context} hold. Indeed, 

\begin{itemize}
\item \underline{$(\mcX,\mcX \cap \omega^\wedge)$ is a left Frobenius pair cut on $\mcS$:} Clearly, $\mcX$ is closed under extensions, cocones and direct summands by definition. Now, by Lemma \ref{lem:relative_equality}, we have $(\mcX \cap \mcS, \mcX \cap \omega^\wedge \cap \mcS) = (\mcX \cap \mcS,\omega \cap \mcS)$ which is a left Frobenius pair in $\mcC$. Thus, it suffices to show that $\omega \cap \mcS$ is $(\mcX \cap \omega^\wedge)$-projective relative generator in $\mcX \cap \omega^\wedge$. Notice first that $\pd_{\omega^\wedge}(\omega \cap \mcS) = 0$ by Lemma \ref{lem:main_lemma}. Now, let $M \in \mcX \cap \omega^\wedge$. Using again Lemma \ref{lem:main_lemma}, there exists an $\mathbb{E}$-triangle $M' \to P \to M\dashrightarrow$ with $P \in \omega \cap \mcS$ and $M' \in \omega^\wedge$. Since $\mcX$ is closed under cocones and $\omega \cap \mcS \subseteq \mcX \cap \mcS$, we get that $M' \in \mcX \cap \omega^\wedge$. Finally, by Remark~\ref{rmk: frob} $\mcX \cap \omega^\wedge$ is closed under 
extensions and direct summands in $\mcC.$ 

\item \underline{$\omega^\wedge \cap \mcS$ is closed under extensions, cones and direct summands:} Notice
first that $\omega\cap \mcS$ is closed under extensions and direct summands by Remark~\ref{rmk: frob}. Then, $\omega^{\wedge}\cap 
\mcS=(\omega\cap \mcS)^{\wedge}$ is closed under extensions and direct summands by Lemma~\ref{lem:main_lemma}-(4).

We see that $\omega^{\wedge}\cap \mcS=(\omega\cap \mcS)^{\wedge}$ is closed under cones. Let $A\to B\to C
\dashrightarrow$ be an $\mathbb{E}$-triangle with $A, B\in (\omega\cap \mcS)^{\wedge}$. The result is clear if 
$\resdim_{\omega\cap \mcS}(B)=0$ (notice that $\omega\cap\mcS$ is closed under isomorphisms since
it is closed under extensions and direct summands). So, we can assume that $\resdim_{\omega\cap \mcS}(B)=:n\geq 1$.
From Lemma~\ref{lem:main_lemma}, we have an $\mathbb{E}$-triangle
$B'\to F\to B\dashrightarrow$ where $F\in \omega\cap \mcS$ and $\resdim_{\omega\cap \mcS}(B')\leq \resdim_{\omega\cap\mcS}(B)-1$. By considering the previous two $\mathbb{E}$-triangles
we get the following commutative diagram from (ET4)$^{op}$
\[
\begin{tikzpicture}[description/.style={fill=white,inner sep=2pt}] 
\matrix (m) [matrix of math nodes, row sep=2.5em, column sep=2.5em, text height=1.25ex, text depth=0.25ex] 
{ 
B' & E & A \\
B' & F & B  \\
{} & C & C,  \\
};  
\path[->] 
(m-1-1) edge (m-1-2) (m-1-2) edge (m-1-3)
(m-2-1) edge (m-2-2) (m-2-2) edge (m-2-3)
(m-1-2) edge (m-2-2) (m-2-2) edge (m-3-2)
(m-1-3) edge (m-2-3) (m-2-3) edge (m-3-3)
; 
\path[-,font=\scriptsize]
(m-3-2) edge [double, thick, double distance=2pt] (m-3-3)
(m-1-1) edge [double, thick, double distance=2pt] (m-2-1)
;
\end{tikzpicture} 
\]
where $E\in (\omega\cap \mcS)^{\wedge}$ since $A, B'$ belong to $(\omega\cap \mcS)^{\wedge}$ and 
$(\omega\cap \mcS)^{\wedge}$ is closed under extensions. Thus, from the second column in the previous diagram
we can deduce that $C\in (\omega\cap\mcS)^{\wedge}$. 

\item \underline{$\omega^\wedge \cap \mcS \subseteq (\mcX \cap \mcS)^\wedge$:} It follows from Lemma \ref{lem:main_lemma}. 
\end{itemize}
Finally, the mappings $\Psi_{\mcS}$ and $\Phi_{\mcS}$ are inverse to each other by Lemma \ref{lem:relative_equality}.
\end{proof}

\begin{remark}\label{rmk: corresp 1}$\,$
In Theorem~\ref{theo:correspondence_1} we can also provide an alternative proof of the fact that $\omega^{\wedge}\cap \mcS$ is closed under extensions, 
direct summands and cones, for any left Frobenius pair $(\mcX, \omega)$ cut on $\mcS$. These arguments are similar to the given
ones in \cite[Thm. 5.4]{BMPS} and \cite[Thm. 4.6]{HMPcut} but in extriangulated categories.

In fact, given a left Frobenius pair $(\mcX, \omega)$ cut on $\mcS$, we know that $(\mcX\cap\mcS, \omega\cap \mcS)$ is a left Frobenius pair in $\mcC$ by definition. Then, from Remark~\ref{rmk: frob1} we get that $(\mcX\cap \mcS)^{\wedge}$ is closed under extensions, direct summands and cones. Moreover, the equalities $$(\omega\cap \mcS)^{\wedge}=
(\mcX\cap \mcS)\cap (\mcX\cap \mcS)^{\perp}=(\mcX\cap \mcS)\cap (\omega\cap \mcS)^{\wedge}$$ hold by \cite[Prop. 3.9]{MDZtheoryAB}. Thus, by proceeding as in \cite[Prop. 4.18]{TGone} and by Lemma~\ref{lem:main_lemma}-(4), we obtain that 
$\omega^{\wedge}\cap \mcS=(\omega\cap \mcS)^{\wedge}$ is 
closed under extensions, direct summands and cones.
\end{remark}

\section{\textbf{Cut $n$-cotorsion pairs in extriangulated categories with higher extension groups}}

In this section we introduce the notion of \emph{cut $n$-cotorsion pair} for extriangulated categories with higher extension groups. This notion
unifies $n$-cotorsion~\cite[Def. 2.1]{HMP} and cut cotorsion~\cite[Def. 2.1]{HMPcut} pairs. 
Throughout this section, $n\geq 1$ denotes a positive integer and $\mcC$ an extriangulated category with higher extension groups.

\begin{definition}\label{def: CnCP ext} Let $\mcS, \mcA, \mcB\subseteq \mcC$. We say that $(\mcA, \mcB)$ is a \textbf{left $n$-cotorsion pair cut on $\mcS$} if 
the following conditions are satisfied.
\begin{enumerate}
\item[(1)] $\mcA$ is closed under direct summands in $\mcC.$
\item[(2)] $\mathbb{E}^{i}(\mcA\cap \mcS, \mcB)=0$ for every $1\leq i\leq n,$
\item[(3)] (Left completeness) For each $S\in \mcS$, there exists an $\mathbb{E}$-triangle
$K\to A\to S\dashrightarrow$ 
with $A\in \mcA$ and $K\in \mcB^{\wedge}_{n-1}$.
\end{enumerate}

Dually, we say that $(\mcA, \mcB)$ is a \textbf{right $n$-cotorsion pair cut on $\mcS$} if $\mcB$ is closed under direct 
summands in $\mcC,$ $\mathbb{E}^{i}(\mcA, \mcB\cap \mcS)=0$ for every $1\leq i\leq n$ and each $S\in \mcS$ admits an $\mathbb{E}$-triangle $S\to B\to C\dashrightarrow$ with $B\in\mcB$ and $C\in \mcA^{\vee}_{n-1}$. Finally, $(\mcA, \mcB)$ is said to be an \textbf{$n$-cotorsion pair
cut on $\mcS$} if it is both a left and right $n$-cotorsion pair cut on $\mcS$. In case there is no need to refer to the subcategory $\mcS$, we shall simply say that $(\mcA, \mcB)$ is a \textbf{(left and/or right) $n$-cotorsion cut}.
\end{definition}

\begin{remark} $\,$
\begin{enumerate}
\item In the case of $1$-cotorsion pairs cut on $\mcS,$ we can remove the condition that $\mcC$ has higher extension groups. Moreover, any $1$-cotorsion pair cut on $\mcS:=\mcC$ is a cotorsion pair in $\mcC$ \cite[Def. 4.1]{Nakaoka1}.

\item Any left $n$-cotorsion pair cut on $\mcS$ with $\mcS:=\mcC$ is a left $n$-cotorsion pair in $\mcC$, for any $n\geq 1$ \cite[Def. 3.1]{HZontherelation}.

\item Any left $1$-cotorsion pair cut on $\mcS$ is a left cotorsion pair cut along $\mcS$, for any $\mcS\subseteq \mcC$ when $\mcC$ is an abelian category~\cite[Def. 2.1]{HMPcut}.
\end{enumerate}
\end{remark}

In the case of the abelian or triangulated categories, the notion of $n$-cotorsion pair cut on $\mcS$ specializes as follows.

\begin{remark} Let $\mcC$ be an abelian (resp., triangulated) category 
and $\mcS, \mcA, \mcB\subseteq \mcC$. In this case,  $(\mcA, \mcB)$ is a \textbf{left $n$-cotorsion pair cut on $\mcS$} if 
the following conditions are satisfied:
\begin{enumerate}
\item[(1)] $\mcA$ is closed under direct summands in $\mcC;$
\item[(2)] $\Ext^{i}_{\mcC}(\mcA\cap \mcS, \mcB)=0$ (resp.,
$\Hom_{\mcC} (\mcA\cap \mcS, \mcB[i])=0$) for every $1\leq i\leq n$;
\item[(3)] For every object $S\in \mcS$, there exists an exact sequence 
$K\rightarrowtail A\twoheadrightarrow S$ (resp., a distinguished triangle $K\to A\to S\to K[1]$) 
with $A\in \mcA$ and $K\in \mcB^{\wedge}_{n-1}$.
\end{enumerate}
\end{remark} 

One source of examples comes from subcategories which are relative quasi-generators and quasi-cogenerators in
other ones. Recall that, for any abelian category
$\mcC$ and $\nu, \mcS\subseteq \mcC$,
$\nu$ is said to be a \emph{relative quasi-generator in
$\mcS$} if for every $S\in \mcS$ there exists an exact sequence $S'\rightarrowtail Y\twoheadrightarrow S$ with $Y\in \nu$ and $S'\in \mcS$. 
Actually, it is easy to prove that, when $\nu$ and $\mcS$ are closed under direct summands and
$\nu$ is an $\mcS$-projective relative quasi-generator in $\mcS$, the pair $(\nu, \mcS)$ is an $n$-cotorsion pair cut on $\mcS$ for every $n\geq 1$. Dually, for
every $n\geq 1$, $(\mcS, \omega)$ is an $n$-cotorsion pair cut on $\mcS$ whenever $\omega$ and $\mcS$ are 
closed under direct summands and $\omega$ is an
$\mcS$-injective relative quasi-cogenerator in $\mcS$. The following example illustrates the mentioned previously. 

\begin{example}\label{ex:carcaj} \cite[Ex. 5.3-(2)]{ZX}
Let $k$ be a field and $\Lambda$ be the quotient path $k$-algebra given by the quiver 
\[
\xymatrix{
1 \ar@/^/[r]^{\alpha} & 2\ar@/^/[l]^{\beta} & 3\ar[l]^{\gamma}
}
\]
with relations $\alpha \beta = 0 = \beta \alpha$. The indecomposable projective $\Lambda$-modules are 
\begin{align*}
P(1) & = \begin{tiny}\begin{array}{c} 1 \\ 2 \end{array} \end{tiny}, & P(2) & = \begin{tiny} \begin{array}{c} 2 \\ 1 \end{array} \end{tiny}, & & \text{and} & P(3) & = \begin{tiny} \begin{array}{c} 3 \\ 2 \\ 1 \end{array} \end{tiny},
\end{align*}
while the indecomposable injective $\Lambda$-modules are given by
\begin{align*}
I(1) & = \begin{tiny}\begin{array}{c} 3 \\ 2 \\ 1 \end{array}\end{tiny}, & I(2) & = \begin{tiny}\begin{array}{ccc} 1 & {} & 3 \\ {} & 2 \end{array}\end{tiny}, & & \text{and} & I(3) & = \begin{tiny}\begin{array}{c} 3 \end{array}\end{tiny}.
\end{align*}
On the other hand, the Auslander-Reiten quiver of $\Lambda$ is
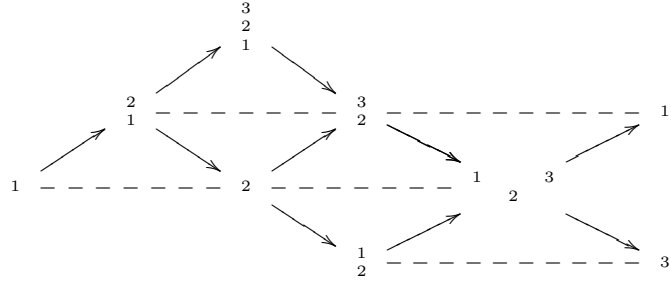
\[
\xymatrix@R=0.3cm{
 & & {\begin{tiny}
\begin{array}{c}
3 \\
2\\
1
\end{array}
\end{tiny}}\ar[rd] & & & \\
 & {\begin{tiny}
\begin{array}{c}
2\\
1
\end{array}
\end{tiny}}\ar[ru] \ar[rd] \ar@{--}[rr] & & {\begin{tiny}
\begin{array}{c}
3\\
2
\end{array}
\end{tiny}}\ar[rd] \ar[rd] \ar@{--}[rr] & & \begin{tiny}
\begin{array}{c}1\end{array}
\end{tiny} \\
\begin{tiny}
\begin{array}{c} 1 \end{array}
\end{tiny} \ar[ru] \ar@{--}[rr] & & \begin{tiny}
\begin{array}{c} 2 \end{array}
\end{tiny} \ar[ru] \ar[rd] \ar@{--}[rr] & & {\begin{tiny}
\begin{array}{ccc}
1& &3\\
&2&
\end{array}
\end{tiny}}\ar[ru] \ar[rd] & \\
 & & & {\begin{tiny}
\begin{array}{c}
1\\
2
\end{array}
\end{tiny}}\ar[ru] \ar@{--}[rr] & & \begin{tiny}
\begin{array}{c} 3 \end{array}
\end{tiny}
}
\]
where the vertices $i$ represent the simple $\Lambda$-module $S(i)$. Consider the class $\mcX = \add(\begin{tiny}
\begin{array}{c}
1
\end{array}
\end{tiny} \oplus \begin{tiny}
\begin{array}{c}
2\\
1
\end{array}
\end{tiny} \oplus \begin{tiny}
\begin{array}{c}
2
\end{array}
\end{tiny} \oplus \begin{tiny}
\begin{array}{c}
1\\
2
\end{array}
\end{tiny})$ of $\Lambda$-modules which are direct summands of finite direct sums of $\begin{tiny}
\begin{array}{c}
1
\end{array}
\end{tiny} \oplus \begin{tiny}
\begin{array}{c}
2\\
1
\end{array}
\end{tiny} \oplus \begin{tiny}
\begin{array}{c}
2
\end{array}
\end{tiny} \oplus \begin{tiny}
\begin{array}{c}
1\\
2
\end{array}
\end{tiny}$. Then, $\mcX$ is a subcategory of $\modu(\Lambda)$ which is closed under extensions. Moreover, $\mcX$ is a Frobenius subcategory of $\modu(\Lambda)$ and the indecomposable projective-injective objects are precisely $\begin{tiny}
\begin{array}{c}
1\\
2
\end{array}
\end{tiny}$ and $\begin{tiny}
\begin{array}{c}
2 \\
1
\end{array}
\end{tiny}$. Thus, the class of projective objects in $\mcX$ is given by $\mcP(\mcX) = \add(P(1) \oplus P(2))$. Notice also that $\mcP(\modu(\Lambda)) = \add(P(1) \oplus P(2) \oplus P(3))$. Setting $(\nu, \mcS):=(\mathcal{P}(\Lambda),\mathcal{X})$ we have that 
$(\nu, \mcS)$ is an $n$-cotorsion pair cut on $\mcS$ for every $n\geq 1$.
An analogous example holds when $\mcC$ is a triangulated category by considering the notions of 
weak-generator and weak-cogenerator in \cite[Def. 5.1]{MendozaSaenzVargasSouto2}.
\end{example}

Other examples come from the theory of stratifying systems \cite{ES} 
and its generalizations \cite{DMS, MendozaSantiagoHomologicalSystems, Valenteexact, ATmixed}. Let
$\mcX$ be a class of $\Lambda$-modules. We
denote by $\mathfrak{F}(\mcX)$ the full subcategory of 
$\modu (\Lambda)$ whose objects are the  $\Lambda$-modules having a finite $\mcX$-filtration; that is, 
 $M$ belongs to $\mathfrak{F}(\mcX)$ if there is a finite chain
$$0=M_{0}\subseteq M_1\subseteq \cdots\subseteq
M_m=M$$
of submodules of $M$ such that $M_i/M_{i-1}$ is 
isomorphic to a $\Lambda$-module in $\mcX,$ for all
$i=1, 2, \ldots, m$.

\begin{definition}\cite[Def. 2.1]{DMS}\label{def: epss} 
An Ext-projective stratifying system (of size $t$) in $\modu(\Lambda)$ consists of a triple $(\Theta, \underline{Q}, \leq),$ where $\Theta=\{\Theta(i)\}_{i=1}^{t}$ is a set of non-zero $\Lambda$-modules, $\underline{Q}=\{Q(i)\}_{i=1}^{t}$ is a set of 
indecomposable $\Lambda$-modules and $\leq$ is a total order on the set $[1,t]:=\{1, 2, \ldots, t\}$ satisfying the following conditions:
\begin{enumerate}
\item $\Hom_{\Lambda}(\Theta(j), \Theta(i))=0$ for $j> i;$
\item for each $i\in [1,t]$, there is an exact sequence
$K(i)\rightarrowtail Q(i)\mathop{\twoheadrightarrow}\limits^{\beta_{i}} \Theta(i)$ in $\modu(\Lambda)$ such that $K(i)\in \mathfrak{F}(\{\Theta(j) : j> i\});$
\item $\Ext_{\Lambda}^{1}(Q, \Theta)=0$, where $Q:=\oplus_{i=1}^{t}Q(i)$.
\end{enumerate}
\end{definition}

The above notion can be extended on the setting of Artin triangulated categories \cite[Def. 3.2]{MendozaSantiagoHomologicalSystems} and it is called $\Theta$-projective systems \cite[Def. 5.2]{MendozaSantiagoHomologicalSystems}. 

\begin{example}\label{ex: sistest}

\underline{$\bullet$ Abelian case:} Let $(\Theta, \underline{Q}, \leq)$ be an Ext-projective stratifying system of size $t$ in $\modu(\Lambda).$ For each $j\in [1,t]$, we consider the class $\mcS_{j}:=\{M\in \mathfrak{F}(\Theta) : \min(M)=j\},$
where $\min(M)$ is defined to be the minimum $k$ such
that $[M: \Theta(k)]\neq 0$ \cite[Lem. 2.6]{DMS}. Notice that $(\mathrm{add}(Q), \mathfrak{F}(\{\Theta(i) : i> j\}))$ is a left $1$-cotorsion pair cut on $\mcS_{j}$ since for every $M\in \mcS_{j}$, there exists a short exact sequence
$
C\rightarrowtail Q'\twoheadrightarrow M
$
with $Q'\in \mathrm{add}(Q)$ and $C\in \mathfrak{F}(\{\Theta(i) : i> j\})$ by \cite[Prop. 2.10]{DMS}.
However this pair is not a right $1$-cotorsion pair cut on $\mcS_j.$ Indeed, if $M\in C_j$ admits an exact sequence 
$M\rightarrowtail C'\twoheadrightarrow Q''$
with $C'\in \mathfrak{F}(\{\Theta(i) : i> j\})$ and $Q''\in \mathrm{add}(Q)$ this
sequence splits due to
$\Ext_{R}^{1}(Q, \mathcal{F}(\Theta))=0$. Thus, 
$M\in \mathfrak{F}(\{\Theta(i) : i> j\})$ 
which is a contradiction.\\

\underline{$\bullet$ Triangulated case:} Let $(\Theta, \underline{Q}, \leq)$ be a $\Theta$-projective system of size $t$ in an Artin triangulated category $\mcC.$ In a similar way as we did in the abelian case, by using \cite[Defs. 5.2 \& 5.12, Cor. 4.4 \& Prop. 6.2-(c)]{MendozaSantiagoHomologicalSystems}, we get that $(\mathrm{add}(Q), \mathfrak{F}(\{\Theta(i) : i> j\}))$ is a left $1$-cotorsion pair cut on $\mathcal{S}_{j},$ 
for each $j\in [1,t].$ In this case, we also have that $(\mathrm{add}(Q)[1], \mathfrak{F}(\{\Theta(i) : i> j\})[1])$ is a 
right $1$-cotorsion pair cut on $\mcS_{j},$ 
for each $j\in [1,t]$.
\end{example}

As we can see in the previous example, from left (respectively, right) cut cotorsion pairs it is sometimes possible to get right (respectively, left) ones depending on the context. Triangulated categories allow us to say more about this duality and its shift functor plays a significant role as we will see in Section~\ref{Sec: applications}.
The following example is the extriangulated version of the previous one.

\begin{example}\label{ex: sistest ext}
Let $k$ be field and $\mcC:=(\mcC, \mathbb{E}, \mathfrak{s}, \mathbb{E}^{-1})$ be a $k$-linear Krull-Schmidt extriangulated category with
negative first extension, and let $(\Theta:=(\Theta(1), \Theta(2), \ldots, 
\Theta(t)), \mathbb{P})$ be a mixed stratifying
system of size $t$ in $\mcC$ \cite[Defs. 2.8 \& 3.13]{ATmixed}. Then,
\begin{enumerate}
\item $\add(\mathbb{P})$ is 
closed under direct summands and 
$\mathbb{E}(\add(\mathbb{P}), \mathfrak{F}(\Theta))=0$ \cite[Props. 3.15 \& 3.9]{ATmixed}.

\item If $\mathfrak{F}(\Theta)$
is $\mathbb{E}$-finite (that is, for each 
$M, N\in \mathfrak{F}(\Theta)$, the $k$-vector
space $\mathbb{E}(M, N)$ is finite dimensional)
and $\Theta(i)$ is a stone, for every 
$i\in [1,t],$ we have that 
$\mathfrak{F}(\Theta)$ admits a finite sequence
of universal extensions by $\Theta$ \cite[Prop. 3.3 \& Def. 3.2]{ATmixed}. Thus, every 
$M\in \mathfrak{F}(\Theta)\setminus {}^{\perp_{\mathbb{E}}}\mathfrak{F}(\Theta)$ admits a 
conflation $K\to P\to M$ with
$P\in \add(\mathbb{P})$ and $K\in \mathfrak{F}(\Theta(\geq j)),$ where $j$ is the minimum value
satisfying $\mathbb{E}(M, \Theta(j))\neq 0$ by
\cite[Prop. 3.10-(a)]{ATmixed}.
\end{enumerate}
Therefore, under the previous hypotheses, $(\add(\mathbb{P}), \mathfrak{F}(\Theta(\geq j))$ is a left $1$-cotorsion pair 
cut on $\mcS_{j}\subseteq \mcC,$ where
$\mcS_{j}$ consists of all $M\in \mathfrak{F}(\Theta)$ such that  
$j$ is the minimum value satisfying that $\mathbb{E}(M, \Theta(j))\neq 0.$
\end{example}

Let $(\mcC, \mathbb{E}, \mathfrak{s})$ be an extriangulated category and 
 $\mathcal{D}\subseteq\mcC$ be closed under extensions and having a zero object of $\mcC.$ Thus, $\mcD$ is closed under isomorphisms and an additive subcategory of $\mcC.$ It is well-known that $\mcD$ is an extriangulated 
category with the extriangulated structure induced by the restriction of $\mathbb{E}$, denoted by
$\mathbb{E}_{\mathcal{D}}$, and the restriction of $\mathfrak{s}$, denoted by $\mathfrak{s}_{\mathcal{D}}$ 
\cite[Rmk. 2.18]{Nakaoka1}. In case that $(\mcA, \mcB)$ is a cotorsion pair cut on $\mcS\subseteq \mcC$
with $\mcA, \mcB\subseteq \mcS$ and $\mcS$ is a subcategory closed under isomorphisms, extensions and direct summands of $\mcC$, the pair
$(\mcA,\mcB)$ can be considered as a complete cotorsion pair in the extriangulated category $\mathcal{S}=(\mathcal{S}, \mathbb{E}_{\mathcal{S}}, \mathfrak{s}_{\mathcal{S}})$ as well. However, in general, a cut cotorsion pair may not be a
complete cotorsion pair in a extriangulated category induced by a subcategory of $\mcC$.

\begin{example}\label{ex: GP, P2} Consider the setting in Example~\ref{ex: GP, P}. 
Notice first that, the existence of the pair $(\mathcal{GP}(\xi), \mathcal{P}(\xi))$ as a left Frobenius pair in $(\mcC, \mathbb{E}_{\xi}, \mathfrak{s}_{\xi})$ implies the one of $(\mathcal{P}(\xi), \mathcal{P}(\xi))$ as a left Frobenius pair in $(\mcC, \mathbb{E}_{\xi}, \mathfrak{s}_{\xi})$ as well (for the closedness conditions see \cite[Def. 4.1]{HZZgorensteinness} and \cite[Prop. 3.24]{Nakaoka1}). Thus, $\mathcal{P}(\xi)^{\wedge}$ is thick and $\id_{\mathcal{P}(\xi)}(\mathcal{P}(\xi)^{\wedge})=0$ from Remark~\ref{rmk: frob1}. We assert the following.

\begin{enumerate}
\item[$\bullet$] \underline{$(\mathcal{GP}(\xi),\mathcal{P}(\xi)^{\wedge})$
is a $1$-cotorsion pair cut on $\mathcal{GP}(\xi)^{\wedge}$ in $(\mcC, \mathbb{E}_{\xi}, \mathfrak{s}_{\xi})$.} Indeed, $\mathcal{P}(\xi)^{\wedge}$ is closed under direct summands and $\mathbb{E}_{\xi}^{i}(\mathcal{GP}(\xi), \mathcal{P}(\xi)^{\wedge})=0$ for 
all $i\geq 1$ by
paragraph above while $\mathcal{GP}(\xi)$ is closed under direct summands by \cite[Thm. 4.17]{HZZgorensteinness}. Moreover, for every $M\in \mathcal{GP}(\xi)^{\wedge}$, there exist 
$\mathbb{E}$-triangles $K\to G\to M\dashrightarrow$ and 
$M\to L\to G'\dashrightarrow$ in $\xi$ such that $G, G'\in \mathcal{GP}(\xi)$ and 
$K, L\in \mathcal{P}(\xi)^{\wedge}$ from \cite[Prop. 5.5]{HZZgorensteinness}.

\item[$\bullet$] \underline{$(\mathcal{P}(\xi), \mathcal{GP}(\xi)^{\perp})$
is a $1$-cotorsion pair cut on $\mathcal{P}(\xi)^{\wedge}$ in $(\mcC, \mathbb{E}_{\xi}, \mathfrak{s}_{\xi})$.} It is clear that $\mathcal{P}(\xi)$ and $\mathcal{GP}(\xi)^{\perp}$ are closed under direct summands. Now, $\mathbb{E}_{\xi}^{i}(\mathcal{P}(\xi), \mathcal{GP}(\xi)^{\perp})=0$ for 
all $i\geq 1$ since
$\mathcal{P}(\xi)\subseteq \mathcal{GP}(\xi)$. Finally, for every $M\in \mathcal{P}(\xi)^{\wedge}$, there exist two
$\mathbb{E}$-triangles $K\to P\to M\dashrightarrow$ and 
$M\to M\to 0\dashrightarrow$ in $\xi$ with 
$P\in \mathcal{P}(\xi)$ and 
$K\in \mathcal{P}(\xi)^{\wedge}\subseteq \mathcal{GP}(\xi)^{\perp}$ 
(see Example~\ref{ex: GP, P}).
\end{enumerate}

From the first pair, it is easy to 
see that $(\mathcal{GP}(\xi),\mathcal{P}(\xi)^{\wedge})$ is a $1$-cotorsion pair cut on $\mathcal{P}(\xi)^{\wedge}$. However, $(\mathcal{GP}(\xi),\mathcal{P}(\xi)^{\wedge})$ can not be considered as cotorsion pair in the
extriangulated category $\mathcal{P}(\xi)^{\wedge}$ since in general the containment $\mathcal{GP}(\xi)\subseteq
\mathcal{P}(\xi)^{\wedge}$ does not hold true \cite[Prop. 5.4]{HZZgorensteinness}.
\end{example}

\begin{proposition}\label{pro: 1cot<->complete cot}
Let $\mcC$ be an extriangulated category and $\mcS\subseteq \mcC$ be closed under extensions and direct summands. Then, $(\mcA, \mcB)$ is a $1$-cotorsion pair cut on $\mcS\subseteq \mcC$ and $\mcA, \mcB\subseteq \mcS$ if, and only if, $(\mcA, \mcB)$ is a cotorsion pair
in the extriangulated category $\mcS$.
\end{proposition}

\begin{proof} It follows from the fact that any coproduct $X\oplus Y$ in  $\mcC,$ with $X,Y\in\mcS,$ belongs to $\mcS$ since $\mcS$ is closed under extensions and direct summands in $\mcC.$
\end{proof}

There are some constructions on extriangulated categories which are neither exact nor triangulated  \cite[Prop. 3.30]{Nakaoka1}. Some of these constructions had been worked, for instance, in \cite{happel1988triangulated, beligiannis2000relative} where it is established that, under certain 
conditions, quotients categories carry triangulated structure. With this in mind, in 
\cite{zhou2018triangulated} is unified different constructions of triangulated quotients categories. 
One of them gives Frobenius categories by considering functorially finite subcategories of a triangulated category
with Auslander-Reiten translation. Recall that a triangulated category $\mcC$ is Frobenius
if  $\mcC$ has enough $\mbE$-projectives and $\mbE$-injectives and $\mathcal{P}_{\mbE}(\mcC)=\mathcal{I}_{\mbE}(\mcC)$. Specifically, they proved the following result. 

\begin{corollary}\cite[Cor. 4.10-(2)]{zhou2018triangulated}
Let $\mcC$ be a triangulated category
with Auslander-Reiten translation $\tau$,
and $\mcX$ a functorially finite 
subcategory of $\mcC$, which satisfies
$\tau \mcX=\mcX$. For any $A, C\in \mcC$, define 
$\mathbb{E}'(C, A):=\mcC(C, A[1])$ to be the collection of all equivalence classes of triangles of the form
$A\mathop{\to}\limits^{f} B\mathop{\to}\limits^{g} C\to A[1]$,
where $f$ is $\mcX$-monic \cite[Def. 3.1]{zhou2018triangulated} and let $\mathfrak{s}'(\delta):=[A\mathop{\to}\limits^{f} B\mathop{\to}\limits^{g} C]$ , for any $\delta\in \mathbb{E}'(C, A)$. Then,
$(\mcC, \mathbb{E}', \mathfrak{s}')$ is a 
Frobenius extriangulated category
whose $\mbE'$-projective objects are precisely $\mcX$.
\end{corollary}

\begin{example}\label{ex: ni exact ni triang}
Let $\mcC$ be a triangulated category
with Auslander-Reiten translation $\tau$,
and $\mcX$ a functorially finite 
subcategory of $\mcC$, which satisfies
$\tau \mcX=\mcX$. Then, the pair $(\mcX, \mcC)$ is an $n$-cotorsion pair cut on $\mcS$
in the extriangulated category $(\mcC, \mathbb{E}', \mathfrak{s}')$,  for any $\mcS\subseteq \mcC$ and
any $n\geq 1$. If in addition $\{0\}\neq \mcX\subsetneqq \mcC$ then $(\mcC, \mathbb{E}', \mathfrak{s}')$ is 
neither exact nor triangulated \cite[Rmk. 4.11]{zhou2018triangulated}.
\end{example}

\section{\textbf{Cut AB-contexts vs. Cut $n$-cotorsion pairs}}

We begin by proving results that involve the concepts seen in Sections 3 and 4. 
We show that, under certain conditions, they are related in some sense. 
These results will be used in the last part of this section to prove the second correspondence theorem of this work.

\begin{proposition}\label{prop: ABcontext->cut1Cot}
Let $\mcC$ be an extriangulated category with higher extension groups, $(\mcA, \mcB)$ be a left weak AB context cut on $\mcS$ and $\omega:=\mcA\cap \mcB$. Then, the following statements hold true.
\begin{enumerate}
\item $(\omega\cap \mcS)^{\wedge}=\mcB\cap \mcS, \id_{\mcA\cap\mcS}(\mcB\cap\mcS)=0$ and $\Thick(\mcA\cap\mcS)=(\mcA\cap\mcS)^{\wedge}.$
\item $(\mcA\cap \mcS, \mcB\cap \mcS)$ is a 
1-cotorsion pair cut on $\Thick(\mcA\cap\mcS)$ with $\id_{\mcA\cap\mcS}(\mcB\cap\mcS)=0$.
\end{enumerate}
\end{proposition}

\begin{proof}
(1) The containment $(\omega\cap \mcS)^{\wedge}\subseteq\mcB\cap \mcS$ is clear since $\mcB\cap \mcS$ is closed under cones. On the other hand, let $B\in \mcB\cap \mcS$. Since $(\mcA\cap \mcS, \omega\cap \mcS)$ is
a left Frobenius pair in $\mcC$ and $\mcB\cap\mcS\subseteq (\mcA\cap\mcS)^{\wedge}$, there is an $\mathbb{E}$-triangle 
$K\to A\to B\dashrightarrow$ with $K\in (\omega\cap \mcS)^{\wedge}\subseteq \mcB\cap\mcS$ and $A\in \mcA\cap\mcS$ from \cite[Thm. 3.7]{MDZtheoryAB}. Moreover, $A\in \omega\cap \mcS$ due to $\mcB\cap\mcS$ is closed under extensions.
Thus, by considering $K\to A\to B\dashrightarrow$ we get that $B\in (\omega\cap\mcS)^{\wedge}$. Hence, 
$(\omega\cap \mcS)^{\wedge}=\mcB\cap \mcS$ holds true. Moreover, from \cite[Lem. 3.8]{MDZtheoryAB} we also have that $\id_{\mcA\cap \mcS}(\mcB\cap\mcS)=\id_{\mcA\cap\mcS}((\omega\cap\mcS)^{\wedge})=
\id_{\mcA\cap\mcS}(\omega\cap \mcS)=0$. Finally, the last equality follows from \cite[Prop. 3.8]{ma2021new}.

(2) Since $\mcA\cap\mcS$ and $\mcB\cap\mcS$ are closed under direct summands (by definition),
$\mathbb{E}(\mcA\cap\mcS, \mcB\cap\mcS)=0$ by the paragraph above and the existence of the $\mathbb{E}$-triangles in
Definition~\ref{def: CnCP ext} follows by \cite[Thm. 3.7]{MDZtheoryAB}, we can conclude that 
$(\mcA\cap\mcS, \mcB\cap\mcS)$ is a 1-cotorsion pair cut on $(\mcA\cap\mcS)^{\wedge}$.
\end{proof}

\begin{lemma}\label{lem:Frobenius_from_cut_cotorsion}
Let $\mcC$ be an extriangulated category with higher extension groups and let $(\mcF,\mcG)$ be a $1$-cotorsion pair cut on
$\Thick(\mcF)$ such that $\id_{\mcF}(\mcG) = 0$. Then,
the following statements hold true.
\begin{enumerate}
\item $(\mcF,\mcF \cap \mcG)$ is a left Frobenius pair in $\mcC$ and $\Thick(\mcF)=\mcF^{\wedge}$.

\item $\mcF \cap \mcG = \mcF \cap \mcF^{\perp_1} = \mcF \cap (\mcF \cap \mcG)^\wedge$.

\item $(\mcF \cap \mcG)^\wedge = \mcF^{\perp} \cap \mcF^\wedge= \mcG \cap \mcF^{\wedge}$.
\end{enumerate}
\end{lemma}

\begin{proof}
(1) From Definition~\ref{def: CnCP ext}, it is easy to see that the equalities 
$$\mcF=\mcF\cap \Thick(\mcF)={}^{\perp_1}\mcG\cap \Thick(\mcF) \mbox{ and }\mcG\cap \Thick(\mcF)=\mcF^{\perp_1}\cap \Thick(\mcF)$$
hold true. By using this, along with condition $\id_{\mcF}(\mcG)=0,$ one can prove that $\mcF$ is closed under extensions and cocones.
Moreover, since $\mcF$ and $\mcG$ are closed under direct summands, so is $\mcF\cap \mcG$. Finally, 
we have $\mcF\cap\mcG$ is an $\mcF$-injective cogenerator
in $\mcF$ and $\id_{\mcF}((\mcF\cap \mcG)^{\wedge})=0$ from Definition~\ref{def: CnCP ext}-(3) and \cite[Prop. 3.9]{MDZtheoryAB}. The equality $\Thick(\mcF)=\mcF^{\wedge}$ follows by \cite[Prop. 3.8]{ma2021new}.

(2) Since $\mcG\cap \Thick(\mcF)=\mcF^{\perp_1}\cap \Thick(\mcF)$ holds, so does $\mcF\cap \mcG=\mcF\cap\mcF^{\perp_1}$. For the equality $\mcF \cap \mcF^{\perp_1} = \mcF \cap (\mcF \cap \mcG)^\wedge$, 
since $\id_{\mcF}((\mcF\cap \mcG)^{\wedge})=0$ we get the inclusion $\mcF \cap \mcF^{\perp_1} \supseteq \mcF \cap (\mcF \cap \mcG)^\wedge$. On 
the other hand, if $M\in \mcF\cap \mcF^{\perp_{1}}$ we can consider an $\mathbb{E}$-triangle 
$M\to G\to F\dashrightarrow$ with $F\in \mcF$ and $G\in \mcG$ which is a split $\mathbb{E}$-triangle
since $M\in \mcF^{\perp_{1}}$. Hence, $M\in \mcG$ due to $\mcG$ is closed under direct summands.

(3) The first equality is due to \cite[Lem. 3.10]{ma2021new}.
On the other hand, the equality $\mcF^{\perp} \cap \mcF^\wedge= \mcG \cap \mcF^{\wedge}$ follows as
the second one in (2).
\end{proof}

\begin{lemma}\label{lem:correspondence_2}
Let $\mcC$ be an extriangulated category with higher extension groups, $\mcS\subseteq \mcC$ be a thick subcategory of $\mcC$ and let $(\mcF,\mcG)$ be a 1-cotorsion pair cut on
$\Thick(\mcF)$ such that $\id_{\mcF}(\mcG) = 0$. If $\mcF \cap \mcG \cap \mcS$ is both a generator and cogenerator in $\mcF \cap \mcG$, then $(\mcF, \mcF \cap \mcG)$ is a left Frobenius pair cut on $\mcS$. 
\end{lemma}

\begin{proof}
From Lemma \ref{lem:Frobenius_from_cut_cotorsion} we get that $(\mcF, \mcF\cap \mcG)$ is a left Frobenius pair in $\mcC$ and so by Remark~\ref{rmk: frob} we have that $\mcF$ is closed under extensions, direct summands and cocones, and $\mcF\cap \mcG$ is
closed under extensions and direct summands. Now, since $\mcF\cap \mcG\cap \mcS$ is a generator in $\mcF\cap \mcG,$ by
assumption and
$\id_{\mcF}(\mcG)=0,$ we have that condition (3) in Definition~\ref{def: cut left Frobenius pair} also holds. Thus, it remains to show that $(\mcF \cap \mcS, \mcF \cap \mcG \cap \mcS)$ is a left Frobenius pair in $\mcC$.
Indeed, $\mcF\cap\mcS$ is closed under extensions, direct summands and cocones since $\mcF$ and 
$\mcS$ satisfy these conditions. In a similar way, since $\mcF, \mcG$ and $\mcS$ are closed under direct summands, so is $\mcF\cap \mcG\cap \mcS$. Now, using that $\id_{\mcF}(\mcG) = 0$, it is only left to show that $\mcF \cap \mcG \cap \mcS$ is a cogenerator in $\mcF \cap \mcS$.
For this, since $(\mcF,\mcG)$ is a
1-cotorsion pair cut on $\Thick(\mcF)$, for any $F \in \mcF \cap \mcS$ there is an $\mathbb{E}$-triangle $F \to G \to F'\dashrightarrow$ with $G \in \mcF \cap \mcG$ and $F' \in \mcF$ as $\mcF$ is closed under extensions. Thus, by using that $\mcF \cap \mcG \cap \mcS$ is a cogenerator in $\mcF \cap \mcG$, we can find an $\mathbb{E}$-triangle
$G \to L \to G'\dashrightarrow$ with $L \in \mcF \cap \mcG \cap \mcS$ and $G' \in \mcF \cap \mcG$. The result follows after applying (ET4) to the $\mathbb{E}$-triangles $F \to G \to F'\dashrightarrow$ and $G \to L \to G'\dashrightarrow$. 
\end{proof}

For a fixed class of objects $\mcS \subseteq \mcC$, we denote by $\mfP_{\mcS}$ the class of pairs $(\mcF,\mcG)$ of classes of objects in $\mcC$ satisfying that $(\mcF, \mcG)$ is a 1-cotorsion pair cut on $\Thick(\mcF)$ with 
$\id_{\mcF}(\mcG) = 0$ and such that
$\mcF \cap \mcG \cap \mcS$ is both a generator and cogenerator in $\mcF \cap \mcG$. We define
a equivalence relation on $\mfP_{\mcS}$ as follows.

\begin{definition}\label{def:equiv rel cut cot}
Let $(\mcF,\mcG), (\mcF',\mcG') \in \mfP_{\mcS}$. We shall say that $(\mcF,\mcG)$ is \textbf{related} to $(\mcF',\mcG')$ in $\mfP_{\mcS}$, denoted $(\mcF,\mcG) \sim (\mcF',\mcG')$, if $\mcF \cap \mcS = \mcF' \cap \mcS$ and $\mcF \cap \mcG \cap \mcS = \mcF' \cap \mcG' \cap \mcS$. 
We denote by $[\mcF,\mcG]_{\mfP_{\mcS}}$ the equivalence class of the representative $(\mcF,\mcG) \in \mfP_{\mcS}$.
\end{definition}

We are now ready to show that there exists a one-to-one correspondence between the quotient classes $\mfP_{\mcS} /\!\!\sim$ and $\mfC_{\mcS} /\!\!\sim$. This result generalizes \cite[Thm. 4.12]{HMPcut}.

\begin{theorem}[Second correspondence theorem]\label{theo:correspondence_2}
Let $\mcC$ be an extriangulated category with higher extension groups and let $\mcS\subseteq \mcC$ be a thick subcategory. Then, the correspondence $\Lambda_{\mcS} : \mfP_{\mcS} /\!\! \sim\,  \to \mfC_{\mcS} /\!\! \sim,\;[\mcF,\mcG]_{\mfP_{\mcS}} \mapsto [\mcF,(\mcF \cap \mcG)^\wedge]_{\mfC_{\mcS}},$ is bijective
with inverse
$\Upsilon_{\mcS}$ given by  
$[\mcA,\mcB]_{\mfC_{\mcS}} \mapsto [\mcA \cap \mcS,\mcB \cap \mcS]_{\mfP_{\mcS}}$.
\end{theorem}

\begin{proof}
First, we see that the mappings $\Lambda_{\mcS}$ and $\Upsilon_{\mcS}$ are well-defined.

On the one hand, if $(\mcF,\mcG)\sim (\mcF',\mcG')$ in $\mfP_{\mcS}$, by Lemma \ref{lem:correspondence_2} and Theorem \ref{theo:correspondence_1}, we have that $(\mcF,(\mcF \cap \mcG)^\wedge)$ and 
$(\mcF',(\mcF' \cap \mcG')^\wedge)$ are left weak AB-contexts cut on $\mcS$. Moreover, 
from Lemma \ref{lem:Frobenius_from_cut_cotorsion}-(2) we get the equalities $\mcF \cap \mcG \cap \mcS = \mcF \cap (\mcF \cap \mcG)^\wedge \cap \mcS$ and $\mcF' \cap \mcG' \cap \mcS = \mcF' \cap (\mcF' \cap \mcG')^\wedge \cap \mcS$ and so $(\mcF,(\mcF \cap \mcG)^\wedge) \sim (\mcF', (\mcF' \cap \mcG')^\wedge)$ in $\mathfrak{C}_{\mcS}$. 
Hence, $\Lambda_{\mcS}$ does not depend on representatives.

On the other hand, let $(\mcA,\mcB) \in \mfC_{\mcS}$. By Proposition~\ref{prop: ABcontext->cut1Cot} we have that $(\mcA \cap \mcS, \mcB \cap \mcS)$ is a 1-cotorsion pair cut on $\Thick(\mcA\cap\mcS)$ satisfying $\id_{\mcA \cap \mcS}(\mcB \cap \mcS) = 0$. Furthermore, it is clear that
$\mcA\cap \mcB\cap\mcS$ is a generator and cogenerator in itself, which implies that
$(\mcA\cap\mcS, \mcB\cap\mcS)\in 
\mathfrak{P}_{\mcS}$. Finally, it is also clear that $\Upsilon_{\mcS}$ does not depend on representatives.

The fact that $\Upsilon_{\mcS}$ is the inverse of $\Lambda_{\mcS}$
follows from Lemma \ref{lem:Frobenius_from_cut_cotorsion} and Proposition~\ref{prop: ABcontext->cut1Cot}.
\end{proof}

\begin{example} Consider Example~\ref{ex: GP, P2}. We know that $\mathcal{GP}(\xi)$ and $\mathcal{P}(\xi)$ are
closed under extensions and cocones by \cite[Cor. 3.23]{Hegorensteinobjects}. By taking $\mathcal{S}:=\mathcal{P}(\xi)^{\wedge},$ we assert that $(\mathcal{GP}(\xi), \mathcal{P}(\xi)^{\wedge})$ and $(\mathcal{P}(\xi), \mathcal{GP}(\xi)^{\perp})$ belong to $\mathfrak{P}_{\mathcal{S}}$. Indeed, it follows due to the fact that the equalities $\mathcal{GP}(\xi)\cap \mathcal{P}(\xi)^{\wedge}=\mathcal{P}(\xi)=\mathcal{P}(\xi)\cap \mathcal{GP}(\xi)^{\perp}$ hold true \cite[Prop. 5.4]{HZZgorensteinness} and $\mathcal{P}(\xi)$
is a generator and cogenerator in itself. Moreover, from the previous equalities, we
get that such pairs are related
in $\mathfrak{P}_{\mathcal{S}}$. 
\[
\xymatrix{
& [\mathcal{GP}(\xi), \mathcal{P}(\xi)^{\wedge}]_{\mathfrak{C}_{\mathcal{S}}}\ar[dr]^{\Upsilon_{\mathcal{S}}} & \\
[\mathcal{GP}(\xi), \mathcal{P}(\xi)^{\wedge}]_{\mathfrak{P}_{\mathcal{S}}}\ar[ur]^{\Lambda_{\mathcal{S}}}
\ar@{=}[rr] &
& [\mathcal{P}(\xi), \mathcal{P}(\xi)^{\wedge}]_{\mathfrak{P}_{\mathcal{S}}}
}
\]
\[
\xymatrix{
& [\mathcal{P}(\xi), \mathcal{P}(\xi)^{\wedge}]_{\mathfrak{C}_{\mathcal{S}}}\ar[dr]^{\Upsilon_{\mathcal{S}}} & \\
[\mathcal{P}(\xi), \mathcal{GP}(\xi)^{\perp}]_{\mathfrak{P}_{\mathcal{S}}}\ar[ur]^{\Lambda_{\mathcal{S}}}
\ar@{=}[rr] &
& [\mathcal{P}(\xi), \mathcal{P}(\xi)^{\wedge}]_{\mathfrak{P}_{\mathcal{S}}}
}
\]
\end{example}

As we can see in the previous example, one equivalence class could have several representatives. In
the following result we prove that the previous correspondence is also valid if we restrict to certain 
representatives of cut cotorsion pairs. For a fixed class $\mcS\subseteq \mcC$, we denote by $\widetilde{\mathfrak{P}}_{\mcS}$ the subclass of $\mathfrak{P}_{\mcS}$
whose elements are $(\mcF,\mcG)\in \mathfrak{P}_{\mcS}$ satisfying $\mcG\subseteq \Thick(\mcF)$.

\begin{corollary}\label{cor:correspondence_3}
Let $\mcC$ be an extriangulated category with higher extension groups and let $\mcS\subseteq \mcC$ be a thick subcategory. Then, the correspondence 
$$\widetilde{\Lambda}_{\mcS} \colon \widetilde{\mfP}_{\mcS} /\!\! \sim \mbox{} \to \mfC_{\mcS} /\!\! \sim,\;[\mcF,\mcG]_{\mathfrak{P}_{\mcS}} \mapsto [\mcF, \mcG]_{\mfC_{\mcS}},$$
 is bijective, with inverse
$\widetilde{\Upsilon}_{\mcS}$  given by 
$[\mcA,\mcB]_{\mfC_{\mcS}} \mapsto [\mcA \cap \mcS,\mcB \cap \mcS]_{\mfP_{\mcS}}$.
\end{corollary}

\begin{proof}
Let $(\mcF,\mcG)\sim (\mcF',\mcG')$ in $\widetilde{\mfP}_{\mcS}$, that is, $\mcF\cap \mcS=\mcF'\cap \mcS$ and $\mcF\cap \mcG\cap \mcS=\mcF'\cap \mcG'\cap \mcS$. By Lemma \ref{lem:correspondence_2} and Theorem \ref{theo:correspondence_1}, we have that $(\mcF,(\mcF \cap \mcG)^\wedge)$ and 
$(\mcF',(\mcF' \cap \mcG')^\wedge)$ are left weak AB-contexts cut on $\mcS$. Moreover, since both pairs are in $\widetilde{\mathfrak{P}}_{\mcS}$ we get that $(\mcF\cap \mcG)^{\wedge}=\mcG$ and $(\mcF'\cap \mcG')^{\wedge}=\mcG'$ from Lemma~\ref{lem:Frobenius_from_cut_cotorsion}-(3). Thus, $(\mcF, \mcG)\sim (\mcF', \mcG')$ in $\mathfrak{C}_{\mcS}$. Therefore, $\widetilde{\Lambda}_{\mcS}$ is well-defined.

Let $(\mcA,\mcB) \in \mfC_{\mcS}$. From Theorem~\ref{theo:correspondence_2}, Definition~\ref{def: cut left AB context}-(3) and Proposition~\ref{prop: ABcontext->cut1Cot}-(1), we have that 
$(\mcA\cap\mcS, \mcB\cap\mcS)\in 
\mathfrak{P}_{\mcS}$, $\mcB\cap \mcS\subseteq (\mcA\cap \mcS)^{\wedge}$ and $(\mcA\cap \mcS)^{\wedge}=\Thick(\mcA\cap \mcS)$, respectively. Hence, 
$(\mcA\cap\mcS, \mcB\cap \mcS)\in \widetilde{\mathfrak{P}}_{\mcS}$. Furthermore, it is clear that 
$(\mcA\cap\mcS, \mcB\cap \mcS)\sim (\mcA'\cap\mcS, \mcB'\cap \mcS)$ in $\widetilde{\mfP}_{\mcS}$ if
$(\mcA, \mcB)\sim (\mcA', \mcB')$ in $\mathfrak{C}_{\mcS}$. Therefore, $\widetilde{\Upsilon}_{\mcS}$ is well-defined.

Finally, we show that they are mutually inverse. In fact,
by definition of $\widetilde{\Lambda}_{\mcS}$ and $\widetilde{\Upsilon}_{\mcS}$, we have that 
$[\mcF, \mcG]_{\mfP_{\mcS}}\mapsto [\mcF, \mcG]_{\mathfrak{C}_{\mcS}} \mapsto [\mcF\cap\mcS, \mcG\cap \mcS]_{\mfP_{\mcS}}$, for any 
$(\mcF, \mcG)\in \widetilde{\mathfrak{P}}_{\mcS}$ and $[\mcA, \mcB]_{\mathfrak{C}_{\mcS}}\mapsto [\mcA\cap \mcS, \mcB\cap \mcS]_{\mfP_{\mcS}}\mapsto 
[\mcA\cap \mcS, \mcB\cap \mcS]_{\mathfrak{C}_{\mcS}}$, for any $(\mcA, \mcB)\in \mathfrak{C}_{\mcS}$. 
In both compositions it is clear that the first and third equivalence class coincide.
\end{proof}

\begin{remark}\label{rmk: S=C} By taking $\mathcal{S}:=\mcC$ in Definitions~\ref{def:Frobenius_and_AB_relations} and~\ref{def:equiv rel cut cot}, the following hold true. 
\begin{enumerate}
\item For any $(\mcX,\omega), (\mcX',\omega')\in 
\mathfrak{F}_{\mathcal{S}}$,
$(\mcX,\omega)\sim (\mcX',\omega')$ if, and only if, 
$\mcX=\mcX'$ and $\omega=\omega'$.

\item For any $(\mcA,\mcB), (\mcA',\mcB')\in 
\mathfrak{C}_{\mathcal{S}}$,
$(\mcA,\mcB)\sim (\mcA',\mcB')$ if, and only if, 
$\mcA=\mcA'$ and $\mcA\cap\mcB=\mcA'\cap\mcB'$ if, and only if, $\mcA=\mcA'$ and $\mcB=\mcB'$ by
Proposition~\ref{prop: ABcontext->cut1Cot}.

\item For any
 $(\mcF,\mcG), (\mcF',\mcG')\in \widetilde{\mathfrak{P}}_{\mathcal{S}}$,
$(\mcF,\mcG)\sim (\mcF',\mcG')$ if, and only if, $\mcF=\mcF'$ and $\mcF\cap \mcG=\mcF\cap\mcG'$ if,
and only if, $\mcG=\mcG'$ by Lemma~\ref{lem:Frobenius_from_cut_cotorsion}-(3). So,
every equivalence class in $\widetilde{\mathfrak{P}}_{\mathcal{S}}/\sim$ 
only has one element.
\end{enumerate}
\end{remark}

\begin{corollary}\cite[Thm. 3.12]{ma2021new}\label{cor: leftFrob<->cotor pairs in Thick}
Let $\mcC$ be an extriangulated category with higher extension groups. Then, the maps
$(\mcX,  \omega)\mapsto (\mcX, \omega^{\wedge}) \mbox{ and } (\mcF, \mcG)\mapsto (\mcF, \mcF\cap \mcG)$
give mutually inverse one-to-one correspondences between the following classes:
\begin{enumerate}
\item Left Frobenius pairs $(\mcX, \omega)$ in $\mcC$.
\item Cotorsion pairs $(\mcF, \mcG)$ in the extriangulated category $\Thick(\mcF)$ with $\id_{\mcF}(\mcG)=0$.
\end{enumerate}
\end{corollary}

\begin{proof}
By taking $\mcS:=\mcC$, by Remark~\ref{rmk: S=C} every equivalence class in $\mathfrak{F}_{\mcS}/\!\!\sim$, $\mathfrak{C}_{\mcS}/\!\!\sim$ and
$\widetilde{\mathfrak{P}}_{\mcS}/\!\!\sim$ has one element. Thus, from Theorem~\ref{theo:correspondence_1} and Corollary~\ref{cor:correspondence_3}, we get that
\begin{align*}
\widetilde{\Upsilon}_{\mcS}\Phi_{\mcS}: \mathfrak{F}_{\mcS}/\!\!\sim\, \to \widetilde{\mfP}_{\mcS}/\!\!\sim 
\quad \mbox{ is given by }\quad (\mcX, \omega)\mapsto (\mcX, \omega^{\wedge}), \\
\Psi_{\mcS}\widetilde{\Lambda}_{\mcS}: \widetilde{\mfP}_{\mcS}/\!\!\sim \to \mathfrak{F}_{\mcS}/\!\!\sim
\quad \mbox{ is given by }\quad (\mcF, \mcG)\mapsto (\mcF, \mcF\cap\mcG).
\end{align*}

Notice that if $(\mcX, \omega)$ is a left Frobenius pair then $(\mcX, \omega^{\wedge})\in \mathfrak{C}_{\mcS}$ and so by definition and Proposition~\ref{prop: ABcontext->cut1Cot}, we obtain
$\omega^{\wedge}\subseteq \mcX^{\wedge}$ and $\id_{\mcX}(\omega^{\wedge})=0$. Since $\Thick(\mcX)$ is
closed under extensions and direct summands in 
$\mcC,$ we get that 
$\Thick(\mcX)$ is an extriangulated category \cite[Rmk. 2.18]{Nakaoka1}. Now, using that $(\mcX, \omega^{\wedge})$ is a $1$-cotorsion pair cut on $\Thick(\mcX)$ with $\mcX, \omega^{\wedge}\subseteq \Thick(\mcX),$ it follows from Proposition~\ref{pro: 1cot<->complete cot} that 
$(\mcX, \omega^{\wedge})$ is a cotorsion pair in the extriangulated category $\Thick(\mcX)$ with $\id_{\mcX}(\omega^{\wedge})=0$. 

Let $(\mcF, \mcG)$ be a cotorsion pair in the extriangulated category $\Thick(\mcF)$ with
$\id_{\mcF}(\mcG)=0$. By Proposition~\ref{pro: 1cot<->complete cot}, we have that
$(\mcF, \mcG)$ is a $1$-cotorsion pair cut on $\Thick(\mcF)$ with $\mcG\subseteq \Thick(\mcF)$ and $\id_{\mcF}(\mcG)=0$. 
Thus, by Corollary~\ref{cor:correspondence_3} we get that $(\mcF, \mcG)\in \mathfrak{C}_{\mcS}$ and so 
$(\mcF, \mcF\cap \mcG)$ is a left Frobenius pair by Theorem~\ref{theo:correspondence_1}.
\end{proof}

\section{\textbf{Connecting several contexts}}\label{Sec: applications}

In the previous sections we established notions and results in the general context that an extriangulated category provides. However all these results can be applied for abelian, exact and triangulated categories. This last section is devoted to show that different contexts can interplay by
using the cut notions mentioned previously.\\ 

For a triangulated category $\mcT$ with shift functor [1], it is possible to get abelian categories
when there exists certain kind of structures. Examples of such structures are bounded $t$-structures for which there exists an associated abelian category  $\mathcal{H}$
called \emph{ the heart} \cite[Lem. 3.2]{bridgeland2007stability}. Among the highlighted properties of
the heart, we have that it is closed under extensions \cite[Lem. 1.3]{lorenzinuniqueness2020} and 
$\mathcal{T}(A, B[n])=0$ for any $A, B\in \mathcal{H}$ and any $n<0$ \cite[Def. 1.2]{lorenzinuniqueness2020}. The following generalizes this notion.

\begin{definition}\cite[Dyer's Thm. A.2]{lorenzin2022compatibility}\label{def: first ext group}
Let $\mathcal{T}$ be a triangulated category and $\mathcal{H}\subseteq \mathcal{T}$ be closed under extensions, having a zero object of $\mathcal{T}$ and satisfying that $\mathcal{T}(A, B[-1])=0,$ for any 
$A, B\in \mathcal{H}$. Then, $\mathcal{H}$ has a natural exact structure, given by defining $A\rightarrowtail B\twoheadrightarrow C$ as a conflation if 
$A\to B\to C\to A[1]$ is a distinguished triangle in $\mathcal{T},$ for some $C\to A[1]$. This association gives 
rise to a natural isomorphism $\Ext^{1}_{\mathcal{H}}(A, B)
\cong \mathcal{T}(A, B[1]),$ for all $A, B\in \mathcal{H}$.
\end{definition}

Definition~\ref{def: first ext group} can be extended
to higher extensions groups. Morever, there is a relation between short exact sequences in $\mathcal{H}$ and distinguished triangles in $\mcT$ as we can see in the following two results.

\begin{proposition}\cite[Thm. A. 7 \& Def. A.12]{lorenzin2022compatibility}\label{pro: A7}
Let $\mathcal{H}$ be an exact subcategory of $\mathcal{T}$ as in Dyer's Theorem A.2. Then there is 
a well-defined map $$f_{n,A,B}:\Ext^{n}_{\mathcal{H}}(A,B)\to \mathcal{T}(A, B[n]),$$ for any $A, B\in \mathcal{H}$ and $n\geq 0$. We say that \emph{$\mathcal{T}$ has all the Ext groups
of $\mathcal{H}$} if the morphism $f_{n,A,B}$ is an
isomorphism for any $A, B\in \mathcal{H}$ and  
$n\in \mathbb{N}$.
\end{proposition}

\begin{lemma}\cite[Lem. 1.9]{lorenzinuniqueness2020}\label{lem: exact<->triang}
Let $\mathcal{T}$ be a triangulated category and $\mathcal{H}\subseteq \mathcal{T}$ be closed under extensions, having a zero object of $\mathcal{T}$ and satisfying that $\mathcal{T}(A, B[-1])=0,$ for any 
$A, B\in \mathcal{H}$. Then,  
$A\mathop{\to}\limits^{f} B\mathop{\to}\limits^{g} C
\mathop{\to}\limits^{h} A[1]$ is a distinguished triangle in
$\mathcal{T}$ if, and only if, $A\mathop{\rightarrowtail}\limits^{f} B
\mathop{\twoheadrightarrow}\limits^{g} C$ is a short exact sequence in $\mathcal{H}$.
\end{lemma}

With this in mind, it is natural to ask when cut notions in $\mathcal{H}$ can be lifted to
cut ones in $\mathcal{T}$. The following result establishes a relation between
the abelian and triangulated context through the notions previously seen. Recall that, given a triangulated category $\mathcal{T}$, a subcategory $\mcS\subseteq \mathcal{T}$
is \emph{resolving} (resp., coresolving) if it is closed under extensions, has a zero object of $\mathcal{T}$ and 
$\mcS[-1]\subseteq \mcS$ (resp., $\mcS[1]\subseteq \mcS$). Equivalently, $\mcS$ is resolving (resp., coresolving) if, and only if, for any
distinguished triangle $U\to V\to W\to U[1]$ (resp., $W\to V\to U\to W[1]$) in $\mathcal{T}$ with
$W\in \mcS$, one has that $U\in \mcS$ if and only if $V\in \mcS$.

\begin{proposition}\label{pro: ABtheory in H}
For a triangulated category $\mathcal{T}$ having all the Ext groups of $\mathcal{H}\subseteq \mathcal{T}$ with $\mathcal{H}$ closed under extensions and direct summands, the following statements hold true.
\begin{enumerate}
\item Let $(\mcX, \omega)$ be a left Frobenius pair in the exact category $\mathcal{H}.$ Then, 
$(\mcX, \omega)$ is a left Frobenius pair in $\mathcal{T}$ if, and only if, $\mcX[-1]\subseteq \mcX$.
\item Let $(\mcA, \mcB)$ be a left weak AB-context in the exact category $\mathcal{H}.$ Then, 
$(\mcA, \mcB)$ is a left weak AB-context in $\mcT$ if, and only if,
$\mcA[-1]\subseteq \mcA$ and $\mcB[1]\subseteq \mcB$.
\item For any $n\geq 1$, if $(\mcF, \mcG)$ is an $n$-cotorsion pair cut on $\mcS$ in the 
exact category $\mathcal{H},$ then so is in $\mathcal{T}$. 
\end{enumerate}
\end{proposition}

\begin{proof}
Closedness by co(cones) follows by the previous reminder, closedness by extensions in $\mathcal{T}$ holds due to the fact that $\mathcal{H}$ has that property  in 
$\mathcal{T}$ and all the classes considered are contained in $\mathcal{H},$ while closedness by 
summands follows as in the proof of Proposition~\ref{pro: 1cot<->complete cot}. 
On the
other hand, orthogonality
relations between the involved classes are a consequence of Proposition~\ref{pro: A7}. Finally,
from Lemma~\ref{lem: exact<->triang} we get 
the existence of the desired distinguished triangles or (co)resolutions depending on the case. 
\end{proof}

For the following result, we recall that a \emph{co-$t$-structure in a triangulated category
$\mathcal{T}$} consists of a pair of subcategories $\mcX, \mcY\subseteq \mcT$ closed under direct
summands and satisfying that $\mcX[-1]\subseteq \mcX$, $\mathcal{T}(\mcX, \mcY)=0$ and $\mathcal{T}=
\mcX*\mcY$ \cite{bondarko2010weight, pauksztello2008compact}.

\begin{theorem}\label{thm: t-str y cut cot}
Let $\mathcal{T}$ be a triangulated category having all the Ext groups of $\mathcal{H}\subseteq \mathcal{T}$ with $\mathcal{H}$ closed under extensions and direct summands. Consider the following classes:
\[
\begin{array}{l}
\mathfrak{A}:=\{(\mcX, \omega) : (\mcX, \omega) \mbox{ is a left Frobenius pair in the exact category } \mathcal{H} 
\mbox{ and } \mcX[-1]\subseteq \mcX\},\\
\mathfrak{B}:=\{(\mcX, \omega) : (\mcX, \omega) \mbox{ is a left Frobenius pair in } \mathcal{T}\},\\
\mathfrak{C}:=\{(\mcA, \mcB) : (\mcA, \mcB) \mbox{ is a co-$t$-structure in the triangulated category } \Thick(\mcA)\},\\
\mathfrak{D}:=\{(\mcA, \mcB) : (\mcA, \mcB) \mbox{ is a cotorsion pair cut on } \Thick(\mcA) 
\mbox{ in } \mcT \mbox{ with }\mcA, \mcB\subseteq \Thick(\mcA)\}.
\end{array}
\]
Then, $\mathfrak{A}=\mathfrak{B}$, $\mathfrak{C}=\mathfrak{D}$ and there exists a one-to-one correspondence between $\mathfrak{B}$ and $\mathfrak{C}.$ 
\end{theorem}

\begin{proof} The equality
$\mathfrak{A}=\mathfrak{B}$ holds true by Proposition~\ref{pro: ABtheory in H}-(1). 

For the
correspondence between $\mathfrak{B}$ and $\mathfrak{C}$, from Corollary~\ref{cor: leftFrob<->cotor pairs in Thick} it suffices to show that
cotorsion pairs $(\mcF, \mcG)$ in the extriangulated category $\Thick(\mcF)$ with $\id_{\mcF}(\mcG)=0$ are co-$t$-structures in the triangulated category $\Thick(\mcF)$.
On the one hand, notice that the correspondence given in Corollary~\ref{cor: leftFrob<->cotor pairs in Thick} sends
left Frobenius pairs to cotorsion pairs $(\mcF, \mcG)$ with the first class
closed under extensions and cocones in a triangulated category. Since closedness by cocones is equivalent to $\mcF[-1]\subseteq \mcF$ it follows that $(\mcF, \mcG)$ is a co-$t$-structure in the 
triangulated category $\Thick(\mcF)$. On the other hand, if $(\mcF, \mcG)$ is a co-$t$-structure in $\Thick(\mcF)$ is easy to see that
it is a cotorsion pair in $\Thick(\mcF)$ with $\id_{\mcF}(\mcG)=0$.

Finally, by the above paragraph, co-$t$-structures in $\Thick(\mcA)$ are cotorsion pairs in the triangulated category $\Thick(\mcA)$ which coincide with $\mathfrak{D}$ by
Proposition~\ref{pro: 1cot<->complete cot}.
\end{proof}

It is known that for co-$t$-structures there is an associated presilting subcategory $\mcS$ called \emph{the coheart}  \cite{aihara2012silting, keller1988aisles, bondarko2010weight, mendoza2010auslander}. The corollary below is an adaptation of a result for
presilting subcategories given in \cite[Lem. 3.3]{pauksztello2020co}.

\begin{proposition}\label{pro: 1-cot<->cot S*S[1]}
Let $\mcT$ be a triangulated category, $\mcS\subseteq \mathcal{T}$ and let $\mcX, \mcY\subseteq \mathcal{T}$ be closed under extensions and direct summands in $\mathcal{T}$ satisfying that
$\mcS\subseteq \mcX$, $\mcS[1]\subseteq \mcY$ and $\mathcal{T}(\mcX, \mcY[1])=0$. 
Then, the
following statements are equivalent.
\begin{enumerate}
\item $(\mcX, \mcY)$ is a $1$-cotorsion pair cut on $\mcS*\mcS[1]$.
\item $(\mcX, \mcY)$ is a left $1$-cotorsion pair cut on $\mcS[1]$.
\item $(\mcX, \mcY)$ is a right $1$-cotorsion pair cut on $\mcS$.
\end{enumerate}
\end{proposition}

\begin{proof}
We prove that (2)$\Rightarrow$(3).
Suppose that $(\mcX, \mcY)$ is a left $1$-cotorsion pair cut on $\mcS[1]$. Under the above conditions it suffices to prove that, 
for every $M\in \mcS,$ there is a distinguished triangle $M\to Y\to X\to M[1],$ with $Y\in \mcY$ and
$X\in \mcX$. However, that fact follows by the left completeness of the pair $(\mcX, \mcY)$ in
Definition~\ref{def: CnCP ext}. From similar reasons, one can justify (3)$\Rightarrow$(2).

For (2)$\Rightarrow$(1), suppose that (2) holds. We assert that $(\mcX, \mcY)$ is a left $1$-cotorsion pair cut on $S*\mcS[1]$. Notice that under the above conditions, we only need to show that 
for every $M\in \mcS*\mcS[1]$ there is 
a distinguished triangle $M[-1]\to Y\to X\to M,$ with $Y\in \mcY$ and
$X\in \mcX$. Indeed, let $M\in \mcS*\mcS[1]$. Then,
there exist two distinguished triangles $S\to M\to R\to S[1]$ and $Y\to X'\to R\to Y[1]$ with $S\in \mcS, R\in \mcS[1], X'\in \mcX$ and $Y\in \mcY$. From the octahedral axiom, we get the 
following commutative diagram of distinguished triangles
\[
\begin{tikzpicture}[description/.style={fill=white,inner sep=2pt}] 
\matrix (m) [matrix of math nodes, row sep=2.5em, column sep=2.5em, text height=1.25ex, text depth=0.25ex] 
{ 
{} & Y & Y & {}\\
S & X & X' & S[1]  \\
S & M & R  & S[1] \\
{} & Y[1] & Y[1], & {}\\
};  
\path[->] 
(m-2-3) edge (m-2-4) (m-3-3) edge (m-3-4)
(m-3-2) edge (m-4-2) (m-3-3) edge (m-4-3)
(m-2-1) edge (m-2-2) (m-2-2) edge (m-2-3) 
(m-3-1) edge (m-3-2) (m-3-2) edge (m-3-3) 
(m-1-2) edge (m-2-2) (m-2-2) edge (m-3-2)
(m-1-3) edge (m-2-3) (m-2-3) edge (m-3-3)
; 
\path[-,font=\scriptsize]
(m-2-4) edge [double, thick, double distance=2pt] (m-3-4)
(m-4-2) edge [double, thick, double distance=2pt] (m-4-3)
(m-1-2) edge [double, thick, double distance=2pt] (m-1-3)
(m-2-1) edge [double, thick, double distance=2pt] (m-3-1)
;
\end{tikzpicture} 
\]
where $S\to X\to X'\to S[1]$ is a distinguished triangle with $X\in \mcX$ since $\mcS\subseteq \mcX$ and
$\mcX$ is closed under extensions. Hence, by considering the distinguished triangle
$M[-1]\to Y\to X\to M$ obtained from
$Y\to X\to M\to Y[-1]$ the assertion follows. Now, by using the equivalence (2)$\Leftrightarrow$(3) and
procceding as above we also get that $(\mcX, \mcY)$ is a right $1$-cotorsion pair
cut on $\mcS*\mcS[1]$. Hence, (1) holds true.
\end{proof}

A remarkable difference between $t$-structures and co-$t$-structures is that 
the coheart of a co-$t$-structure in general may not be abelian. However, 
\emph{the extended coheart} of a co-$t$-structure is an extriangulated category with 
the extriangulated  structure induced by the triangulated structure of $\mathcal{T}$ \cite[Lem. 1.6]{pauksztello2020co}. Moreover, given a co-$t$-structure in a triangulated category, there
is a one-to-one correspondence with co-$t$-structures ``sufficiently close'' to the initial one 
and cotorsion pairs in its extended coheart \cite[Thm. 2.1]{pauksztello2020co}. By using an alternative description of the extended coheart \cite[Lem. 2.1]{iyama2014intermediate}, we get the following result.

\begin{theorem}\label{thm: co-t-st y cut cot}
Let $\mcT$ be a triangulated category, $(\mcA, \mcB)$ be a co-$t$-structure in $\mcT$, $\mcS:=\mcA[1]\cap \mcB$ be the coheart of $(\mcA, \mcB)$ and 
$\mathcal{Z}:=\mcA[2]\cap \mcB$ be the extended coheart of $(\mcA, \mcB)$. Consider the following classes:

\[
\begin{array}{l}
\mathfrak{A}:=\{(\mcA', \mcB') : (\mcA',\mcB') \mbox{ is a co-$t$-structure in }
\mcT \mbox{ with } \mcA\subseteq \mcA'\subseteq \mcA[1]\},\\
\mathfrak{B}:=\{(\mcX, \mcY) : (\mcX,\mcY) \mbox{ is a cotorsion pair in the extriangulated category }
\mathcal{Z}\},\\
\mathfrak{C}:=\{(\mcX, \mcY) : (\mcX,\mcY) \mbox{ is a $1$-cotorsion pair cut on }
\mathcal{Z} \mbox{ in } \mcT \mbox{ with } \mcX, \mcY\subseteq \mathcal{Z}\},\\
\mathfrak{D}:=\{(\mcX, \mcY) : (\mcX,\mcY) \mbox{ is a left $1$-cotorsion pair cut on }
\mcS[1] \mbox{ in } \mcT \mbox{ with } \mcX, \mcY\subseteq \mathcal{Z}\},\\
\mathfrak{E}:=\{(\mcX, \mcY) : (\mcX,\mcY) \mbox{ is a right $1$-cotorsion pair cut on }
\mcS \mbox{ in } \mcT \mbox{ with } \mcX, \mcY\subseteq \mathcal{Z}\}.\\
\end{array}
\]
Then, there is a one-to-one correspondence between $\mathfrak{A}$ and $\mathfrak{B}$ and
$\mathfrak{B}=\mathfrak{C}=\mathfrak{D}=\mathfrak{E}$.
\end{theorem}

\begin{proof}
The correspondence between $\mathfrak{A}$ and $\mathfrak{B}$ follows by \cite[Thm. 2.1]{pauksztello2020co}.
On the other hand, we have from Proposition~\ref{pro: 1cot<->complete cot} that
the equality $\mathfrak{B}=\mathfrak{C}$ holds true. The remaining equalities hold after noticing that $\mathcal{Z}=\mcS*\mcS[1]$ \cite[Lem. 2.1]{iyama2014intermediate} and the equivalence in Proposition~\ref{pro: 1-cot<->cot S*S[1]} is also valid if we add the condition $\mcX, \mcY\subseteq \mathcal{Z}$ in every
statement. 
\end{proof}


\section*{\textbf{Acknowledgements}}


\section*{\textbf{Funding}}

The first author thanks Programa de Becas Posdoctorales DGAPA-UNAM.


\bibliographystyle{plain}
\bibliography{bibliohmp17}
\end{document}